%% file: main.tex
\documentclass{amsart}

\usepackage{color,hyperref,mathrsfs,amssymb,braket}
\usepackage{amsrefs} 
\usepackage{epsfig}
\usepackage{makecell} 

\theoremstyle{plain}
\newtheorem{theorem}{Theorem}[section]
\newtheorem{lemma}[theorem]{Lemma}

\newtheorem{corollary}[theorem]{Corollary}
\theoremstyle{definition}
\newtheorem{definition}[theorem]{Definition}

\theoremstyle{remark}
\newtheorem{remark}[theorem]{Remark}

\numberwithin{equation}{section}

\newcommand{\abs}[1]{\lvert#1\rvert}

\newcommand{\Lr}[1]{\left(#1\right)}
\newcommand{\lr}[1]{\Bigl(#1\Bigr)}

\newcommand{\nm}[2]{\|\,#1\,\|_{#2}}

\newcommand{\mc}[1]{\mathcal{#1}}
\newcommand{\mb}[1]{\mathbb{#1}}
\newcommand{\ms}[1]{\mathscr{#1}}
\newcommand{\mr}[1]{\mathrm{#1}}

\newcommand{\wh}[1]{\widehat{#1}}

\begin{document}
\title[Sharp uniform approx. for spectral Barron functions by DNN]{Sharp uniform approximation for spectral Barron functions by deep neural networks}
\author[Y. L. Liao]{Yulei Liao}
\address{Department of Mathematics, Faculty of Science, National University of Singapore, 10 Lower Kent Ridge Road, Singapore 119076, Singapore}
\email{ylliao@nus.edu.sg}
\author[P. B. Ming \and H. Yu]{Pingbing Ming \and Hao Yu}
\address{SKLMS, Institute of Computational Mathematics and Scientific/Engineering Computing, AMSS, Chinese Academy of Sciences, Beijing 100190, China}
\address{School of Mathematical Sciences, University of Chinese Academy of Sciences, Beijing 100049, China}
\email{mpb@lsec.cc.ac.cn, yuhao@amss.ac.cn}
\thanks{The work of Ming is supported by the robotic AI-Scientist platform of the Chinese Academy of Sciences. This work was also funded by National Natural Science Foundation of China through Grant No. 12371438.}
\date{\today}

\begin{abstract}
This work explores the neural network approximation capabilities for functions within the spectral Barron space $\mathscr{B}^s$, where $s$ is the smoothness index. We demonstrate that for functions in $\ms{B}^{1/2}$, a shallow neural network (a single hidden layer) with $N$ units can achieve an $L^p$-approximation rate of $\mathcal{O}(N^{-1/2})$. This rate also applies to uniform approximation, differing by at most a logarithmic factor. Our results significantly reduce the smoothness requirement compared to existing theory, which necessitate functions to belong to $\mathscr{B}^1$ in order to attain the same rate. Furthermore, we show that increasing the network's depth can notably improve the approximation order for functions with small smoothness. Specifically, for networks with $L$ hidden layers, functions in $\mathscr{B}^s$ with $0 < sL \le 1/2$ can achieve an approximation rate of $\mathcal{O}(N^{-sL})$. The rates and prefactors in our estimates are dimension-free. We also confirm the sharpness of our findings, with the lower bound closely aligning with the upper, with a discrepancy of at most one logarithmic factor.
\end{abstract}
\maketitle

\input{Section/introduction}
\input{Section/Shallow_NN_Lp_approximation}

\input{Section/Lp_approximation}
\input{Section/Linf_approximation}
\input{Section/Lower_bound}

\section{Conclusion}
We provide a comprehensive analysis of neural network approximation capabilities for spectral Barron functions, significantly lowering the smoothness requirements compared to existing theories. By establishing approximation rates for both shallow and deep networks, we show that functions with small smoothness can achieve substantial approximation efficiency. Our findings indicate that shallow networks attain Monte Carlo rates for functions in $\mathscr{B}^{1/2}$, while deeper networks further reduce the necessary smoothness requirements. The sharpness of the approximation rates is confirmed, with lower bounds matching the upper bounds, up to a logarithmic factor. Future research could explore higher-order convergence results for spaces with greater smoothness.
\appendix
\input{Appendix/bound_for_Gamma_function}

\bibliography{references}
\end{document}

%% file: Section/introduction.tex
\section{Introduction}\label{Section:intro}
In recent years, machine learning has achieved remarkable advancements, revolutionizing numerous fields with its innovative applications. Part of such success lies in the neural network's capability to effectively approximate particularly high-dimensional functions of interest. It is not our intention to provide a comprehensive review of this field, and we refer to \cites{Pinkus1999, DeVore2021NNapproximation} for a state of the art overview.

Originating from Barron's seminal work~\cites{Barron:1992,Barron:1993,Barron:1994}, the Barron class has been widely employed to study the approximation and generalization properties of the neural networks. 
Specifically, for a cube $\Omega:=[0,1]^d$ and a real-valued function $f:\Omega\to\mb{R}$, if the spectral norm $\upsilon_{f,s}$ is finite, then $f$ belongs to the Barron class. Here the spectral norm $\upsilon_{f,s}$ with $s\ge 0$ is defined by
\[
	\upsilon_{f,s}{:}=\inf_{f_e|_\Omega=f}\int_{\mathbb{R}^d}\abs{\xi}^s_1\abs{\wh{f}_e(\xi)}\mathrm{d}\xi,
\]
where $\abs{\cdot}_1$ is the $\ell^1$-norm for any $\mb{R}^d$-vector, $f_e:\mb{R}^d\to\mb{R}$ is any extension of $f$ from $\Omega$ to $\mb{R}^d$, and $\widehat{f}_e$ is the Fourier transform of $f_e$ in the sense of distribution. 
The spectral Barron space $\ms{B}^s(\Omega)$ is the Banach space which consists of all $f \in \ms{S}'(\Omega)$ such that 
\[
\nm{f}{\ms{B}^s(\Omega)}{:}=\inf_{f_e|_\Omega=f}\int_{\mathbb{R}^d}(1+\abs{\xi}_1)^{s}\abs{\wh{f}_e(\xi)}\mr{d}\xi.
\]
is finite.
This kind of spaces may be traced back to a particular case of H\"ormander space~\cites{Hormander:1963,DGSM60:2014}, often referred to as the Fourier Lebesgue space~\cite{Pilipovic:2010}. 
The index $s$ reflects the smoothness of functions and $\ms{B}^s(\Omega)\hookrightarrow\ms{B}^t(\Omega)$ for $s\ge t$. 
A comprehensive study on such space concerning the completeness and the embedding results was conducted in~\cite{LiaoMing:2023}, 
which proved that $\ms{B}^s(\Omega)\hookrightarrow C^s(\Omega)$, where $C^s(\Omega)$ is the H\"older space with exponent $s$. 
Our example in Appendix~\ref{AppendixB} illustrate that generally, the H\"older continuity of $\ms{B}^s$-functions can not be higher than $s$.
Therefore, the target function space in this paper, $\ms{B}^s(\Omega)$ with small $s$, may contain many rough functions.  

An outstanding feature of the spectral Barron space is that its functions can be approximated by neural networks without suffering from the curse of dimensionality. Specifically, functions in $\ms{B}^1(\Omega)$ can be approximated by shallow neural networks (SNN) with $N$ sigmoidal units, resulting in an error bound $\mathcal{O}(N^{-1/2})$; see~\cite{Barron:1993}. Studies in~\cites{Siegel:2020,Xu:2020} have generalized this to $\ms{B}^{k+1}(\Omega)$-functions in the $H^k$-norm with ReLU units. Furthermore, improved $L^2$-approximation rates have been achieved in~\cites{Bresler:2020,Siegel:2022} for $\ms{B}^s(\Omega)$-functions with large $s$. Most of these works necessitate that the target function belongs to $\ms{B}^s(\Omega)$ with $s\ge 1$ to attain the Monte Carlo rate of $\mc{O}(N^{-1/2})$. 
The initial demonstration that $\ms{B}^{1/2}(\Omega)$-functions are sufficient to reach the  $\mathcal{O}(N^{-1/2})$ rate in $L^2$-norm was presented in \cite{Siegel:2022}, utilizing a multiscale expansion, though the prefactor in this case may depend exponentially on the dimension. 
This exponential dependence was removed in \cite{LiaoMing:2023}, where a universal constant was introduced.

For general $p$ with $1\le p<\infty$, \textsc{Makovoz}~\cite{Makovoz1996Random} first established the SNN $L^p$-approximation rate $\mc{O}(N^{-1/2-1/(p^*d)})$ for $\ms{B}^1(\Omega)$-functions, where $p^\ast$ is the minimal even integer satisfying $p^\ast\ge p$. $L^{p}$-approximation rate $\mc{O}(N^{-\min(s(p),1/2)})$ for Barron-type functions with smoothness index $s(p)=1-1/p$ has been showed in~\cite{MengMing:2022}, while the target function space therein is slightly smaller than ours. Working in the variation spaces of SNNs, the authors 
in~\cite{Donahue:1997}
and~\cite{SiegelXu:2022} investigated $L^{p}$-approximation of the convex hull of a dictionary by convex combinations.

As to the uniform approximation, the original $\mc{O}(N^{-1/2})$ approximation rate for  $\ms{B}^1(\Omega)$-functions using SNNs with Heaviside activation function has been established by \textsc{Barron} in~\cite{Barron:1992}. The same rate has been proved in~\cite{Yukich1995Supnorm}  for $\ms{B}^{2m+1}(\Omega)$-functions in the $W^{m,\infty}$-norm. Given the practical preference to ReLU networks, an $\mc{O}(N^{-1/2-1/d}\sqrt{\ln N})$ rate for $\ms{B}^{2}(\Omega)$-functions has been achieved in~\cite{Barron:2018}. Recently, an $\mc{O}(N^{-1/2})$ approximation rate for $\ms{B}^{1}(\Omega)$-functions have been obtained in~\cite{Caragea:2023}. 
Improved approximation rates for $\ms{B}^s(\Omega)$-functions with large $s$ have been derived in~\cite{MaSiegelXu:2022}. 
Sharp uniform approximation rates have also been proved for variation space corresponding to shallow $\operatorname{ReLU}^{k}$ networks.  
We refer to~\cites{Bach2017Breaking,MaSiegelXu:2022,SiegelXu:2022, Siegel2025optimal} for more details.

The distinction between deep neural networks (DNNs) and SNNs has been enlightened in the depth separation theorems proved in~\cites{Eldan:2016,Telgarsky:2016,Shamir:2022}, which highlighted the greater expressive power of DNNs. Only a limited number of works have addressed the problem of DNN approximation for spectral Barron functions. The pioneering work~\cite{BreslerNagaraj:2020} established an $L^2$-approximation rate of $\mathcal{O}(N^{-sL/2})$ for $\ms{B}^{s}(\Omega)$-functions using a ReLU network with $L$ hidden layers and $N$ units per layer, where $0 < sL \leq 1$. However, this result is not optimal, even for SNN, i.e. $L=1$. The improved approximation rate to $\mathcal{O}(N^{-sL})$ for $\ms{B}^{s}(\Omega)$-functions with $0 < sL \leq 1/2$ has recently been proved in~\cite{LiaoMing:2023}. Despite these advancements, the $L^p$-approximation for general $p$ remains an open problem.

It is worth mentioning the approximation results for neural networks targeting analytic function~\cite{E2018analytic}, band-limited function~\cites{Du:2019,Montanelli2021bandlimited}, mixed derivative Besov spaces~\cites{Suzuki:2021,Bolcskei:2021}, and H\"older spaces~\cites{Yarotsky2017error,Yarotsky2018,Yarotsky2020phase,LuShenYangZhang:2021,Shen2022Optimal,Jiao2023Holder}. These studies collectively contribute to a broader understanding of the approximation capabilities of neural networks in various function spaces.

\begin{table}[htbp]\footnotesize
\caption{Approximation rates for functions in $\ms{B}^s(\Omega)$ by NNs with $L$ hidden layers and $N$ units each layer, where $0 <sL\le 1/2$. Only universal constants are implied.}\label{Table: Approximation rates proved in this paper} 
\begin{tabular}{cccc}
\hline\noalign{\smallskip}
   & $L^2(\Omega)$-norm & $L^{p}(\Omega)$-norm, $p>2$  & $L^{\infty}(\Omega)$-norm  \\
\noalign{\smallskip}\hline\noalign{\smallskip}
SNN ($L = 1$) & $2^d\nm{f}{\ms{B}^{1/2}}N^{-1/2}$~\cite{Siegel:2022} & $\sqrt{pd}\nm{f}{\ms{B}^{1-1/p}_2} N^{-1/2}$~\cite{MengMing:2022}{*} & $\sqrt{d}\,\upsilon_{f,1}N^{-1/2}$~\cites{Barron:1992,Caragea:2023} \\
\noalign{\smallskip}\hline\noalign{\smallskip}
DNN & $\nm{f}{\ms{B}^{2s}}N^{-sL}$~\cite{BreslerNagaraj:2020} & -- & --\\
\noalign{\smallskip}\hline\noalign{\smallskip}
Our works & $\upsilon_{f,s} N^{-sL}$~\cite{LiaoMing:2023} & $\sqrt{p}\,\upsilon_{f,s} N^{-sL}$ &  $\sqrt{dL}\,\upsilon_{f,s}N^{-sL}\sqrt{\ln N}$\\
\noalign{\smallskip}\hline\noalign{\smallskip}
Lower bound & $\upsilon_{f,s} N^{-sL}$~\cite{LiaoMing:2023} & \multicolumn{2}{c}{$\upsilon_{f,s} N^{-sL}$}  \\
\noalign{\smallskip}\hline
\end{tabular}
*The space $\ms{B}^{1-1/p}_2\hookrightarrow\ms{B}^{1-1/p}$ with norm $\nm{f}{\ms{B}^{1-1/p}_2}:=\inf_{f_e|_\Omega=f}(\nm{f_e}{L^2(\mb{R}^d)}+\upsilon_{f_e,1-1/p})$.
\end{table}

In this work, we investigate the non-asymptotic $L^{p}$-approximation by DNN for $\ms{B}^s(\Omega)$-functions with small $s$. For $1\le p<\infty$, we establish in Theorem~\ref{Theorem: Lp approximation by deep NN} and Corollary~\ref{Corollary: approximation bounds with spectral semi-norm DNN} that for $\ms{B}^s(\Omega)$-functions with $0<sL\le 1/2$, a deep ReLU network with $L$ hidden layers and $N$ units per layer achieves an $L^p$-approximation rate of $\mc{O}(N^{-sL})$, 
with an $\mc{O}(\sqrt{p})$ prefactor, which blows up as $p$ tends to infinite. Under the same circumstances, we derive an uniform approximation rate of $\mc{O}(N^{-sL} \sqrt{\ln N})$ in Theorem~\ref{Theorem: Linf approximation by deep NN} and Corollary~\ref{Corollary: Linf approximation bounds with spectral semi-norm DNN}, with an $\mc{O}(\sqrt{dL})$ prefactor. These bounds are sharp, as evidenced by the matching lower bound of $\mc{O}(N^{-sL})$ proved in Theorem~\ref{Theorem: lower bound for DNN approximation}. 
Our results, as summarized in Table~\ref{Table: Approximation rates proved in this paper}, extend the $L^{2}$-approximation bound in \cite{LiaoMing:2023} and provide optimal convergence rates compared to known results. For the sake of comparison, we list all these results in Table~\ref{Table: Approximation rates proved in this paper}. 

The remainder of the paper is organized as follows. In Section~\ref{Section: L^p approximation by shallow NN}, we establish the $L^p$-approximation results for shallow sigmoidal networks. In Section~\ref{Section: deep neural network approximation}, we present the approximation bounds for deep ReLU networks and provide matching lower bounds to demonstrate the sharpness of our results. A proof of a technical lemma and an example showcasing the roughness of $\ms{B}^{s}$-functions are deferred to Appendix~\ref{Appendix Section: Bound for Gamma function} and Appendix~\ref{AppendixB}, respectively.

%% file: Section/Shallow_NN_Lp_approximation.tex
\section{Shallow sigmoidal networks approximation} \label{Section: L^p approximation by shallow NN}
In this part, we establish the $L^p$-approximation bound stated in Theorem~\ref{Theorem: shallow NN Lp approximation} and Corollary~\ref{Corollary: approximation bounds with spectral semi-norm shallow NN} for SNN with general sigmoidal activation functions. 
\begin{definition}
A sigmoidal function is a bounded function $\sigma:\mb{R}\to\mb{R}$ such that
	\[
		\lim_{t\to-\infty}\sigma(t)=0,\qquad \lim_{t\to\infty}\sigma(t)=1.
	\]
\end{definition}

An example of a sigmoidal function is the Heaviside function $\chi_{[0,\infty)}$, where $\chi_I$ denotes the indicator function\footnote{The indicator function $\chi_I$ equals $1$ on $I$ and $0$ otherwise.}. The following lemma shows that $\chi_{[0,\infty)}$ may be viewed as the limit case of any sigmoidal function. 
It is worthwhile to mention that the gap between sigmoidal function and $\chi_{[0,\infty)}$ cannot be dismissed in the uniform norm, as pointed out by~\cite{Caragea:2023}.
\begin{lemma}\label{lema:sigmoidal}
	Let $1\le p <\infty$, for fixed $\omega\in\mb{R}^d\backslash\{0\}$ and $b\in\mb{R}$,
	\[
		\lim_{\tau\to\infty}\nm{\sigma(\tau(\omega\cdot x+b))-\chi_{[0,\infty)}(\omega\cdot x+b)}{L^p(\Omega)}=0.
	\] 
\end{lemma}

\begin{proof}
By definition, for any $\varepsilon>0$, there exists $\delta>0$ large enough such that 
\[
	\abs{\sigma(t)-\chi_{[0,\infty)}(t)}<\varepsilon,  \quad \text{when $\abs{t}>\delta$}.
\]
We divide the cube $\Omega$ into $\Omega_1{:}=\set{x\in\Omega|\abs{\tau(\omega\cdot x+b)}<\delta}$ and $\Omega_2{:}=\Omega\setminus\Omega_1$ with a proper choice of $\tau>0$ later on. We change the variable $t=(\omega\cdot x+b)/\abs{\omega}$, and for any fixed $t$, we recall the cube slicing lemma of \textsc{Ball}~\cite[Theorem 4]{Ball:1986cube} to obtain $\abs{\Omega_t}\le\sqrt{2}$, where $\Omega_t:=\set{x\in\Omega|\omega\cdot x+b=\abs{\omega}t}$. This gives
\begin{align*}
\int_{\Omega_1}\abs{\sigma(\tau(\omega\cdot x+b))-\chi_{[0,\infty)}(\omega\cdot x+b)}^p\mr{d}x&\le\int_{-\delta/(\tau\abs{\omega})}^{\delta/(\tau\abs{\omega})}\abs{\sigma(t)-\chi_{[0,\infty)}(t)}^p\abs{\Omega_t}\mr{d}t\\
&\le\frac{2\sqrt{2}\delta}{\tau\abs{\omega}}(1+\nm{\sigma}{L^\infty(\mb{R})})^p\\
&\le\varepsilon^p,
\end{align*}
where we have chosen $\tau\ge2\sqrt{2}\delta(1+\nm{\sigma}{L^\infty(\mb{R})})^p/(\varepsilon^p\abs{\omega})$ in the last step, and
	\[
	\int_{\Omega_2}\abs{\sigma(\tau(\omega\cdot x+b))-\chi_{[0,\infty)}(\omega\cdot x+b)}^p\mr{d}x\le\varepsilon^p.
	\]
A combination of the above two inequalities gives the desired result.
\end{proof}

In what follows, we call a shallow sigmoidal neural network with $N$ units as 
\begin{equation*} \label{eq: shallow sigmoidal neural network with N units}
f_{N}(x) = \sum_{i = 1}^{N} c_{i} \sigma(w_{i}\cdot x + b_{i})+c_{0}, \quad w_{i}\in \mathbb{R}^{d},\  b_{i},c_{i}\in \mathbb{R},
\end{equation*}
where $\sigma$ is a sigmoidal function.
In~\cite{Makovoz1996Random}, Makovoz has proved the following $L^{p}$-approximation bound 
for $f\in\ms{B}^{1}(\Omega)$ with shallow sigmoidal networks.
\begin{lemma}[{\cite[Theorem 3]{Makovoz1996Random}}] \label{Lemma: shallow sigmoidal Lp app B^1 by Makovoz1996}
Let $p\ge 1$ and $f \in \ms{B}^{1}(\Omega)$. There exists a shallow sigmoidal neural network with $N$ units such that 
$$\|f-f_N\|_{L^p(\Omega)} \le C\sqrt{p} \upsilon_{f,1} N^{-1/2-1/(p^{*}d)},$$
where $C$ is a universal constant, and $p^{\ast}$ is the minimal even integer satisfying $p^{\ast}\ge p$.
\end{lemma}

Compared to the Monte Carlo convergence rate $\mathcal{O}(N^{-1/2})$, the improvement of the approximating rate is significant for small $d$, while disappears when $d\to \infty$.

Now we are ready to present the $L^{p}$-approximation by SNN. We follow the framework of~\cites{BreslerNagaraj:2020, Siegel:2022} and leverage a multiscale expansion representation firstly used for SNN approximation in~\cite{Siegel:2022}, while we achieve an improved convergence rate and the prefactor is dimension-free. 
\begin{theorem}\label{Theorem: shallow NN Lp approximation}
Let $p\ge 2$ and $f\in\ms{B}^s(\Omega)$ with $0<s\le 1/2$. For any  $N\in \mathbb{N}_+$, there exists a shallow sigmoidal network with $N$ units  such that 
\begin{equation}\label{ineq: shallow NN Lp approximation bound}
\|f-f_N\|_{L^p(\Omega)} \le 28 \pi \sqrt{p} \|f\|_{\ms{B}^s(\Omega)} N^{-s}. 
\end{equation}
\end{theorem}
\begin{remark}
Theorem~\ref{Theorem: shallow NN Lp approximation} may be extended to other commonly used activation functions such as Hyperbolic tangent, SoftPlus, ELU, Leaky ReLU, $\operatorname{ReLU}^k$, just name a few of them, because all these activation functions can be reduced to sigmoidal functions by certain shifting and scaling arguments; e.g., for SoftPlus, we observe that $\operatorname{SoftPlus}(t)-\operatorname{SoftPlus}(t-1)$ is a sigmoidal function.
\end{remark}
To prove Theorem~\ref{Theorem: shallow NN Lp approximation}, we firstly perform a multiscale expansion of the target function and rewrite it as an expectation. Secondly, we employ an appropriate Monte Carlo discretization with reduced variance. Thirdly, symmetrization and Khintchine inequality will be exploited to estimate the approximation error, and the number of Heaviside neurons used to represent Monte Carlo discretization will be counted. Finally, by Lemma~\ref{lema:sigmoidal}, we generalize the approximation by Heaviside networks to general sigmoidal neural networks.

As a preparation, we firstly introduce a multiscale decomposition of a $W^{1,\infty}(\mb{R})$-function, which is a reformulation of~\cite[Lemma 3]{Siegel:2022} with a slightly tighter bound on the coefficient. 
\begin{lemma}\label{lema:expansion}
	Let $g\in W^{1,\infty}(\mb{R})$, and define
	\begin{equation}\label{eq:gm}
		g_m(t)=\sum_{l=0}^m\sum_{j=-\infty}^\infty\alpha_{l,j}\chi_{[0,1)}(2^lt-j)
	\end{equation}
	with $\alpha_{0,j}=g(j)$ and $\alpha_{l,j}=g(2^{-l}j)-g(2^{1-l}\lfloor j/2\rfloor)$ for $l\ge 1$. Then $g_m\to g$ in $L^\infty(\mb{R})$. Moreover, $\abs{\alpha_{l,j}}\le 2^{-l}\max(\nm{g}{L^\infty(\mb{R})},\nm{g'}{L^\infty(\mb{R})})$.
\end{lemma}

\begin{proof}
The proof falls naturally into three parts. Firstly, we claim 
\begin{equation}\label{eq:claim}
		g_m(t)=\sum_{j=-\infty}^\infty g(2^{-m}j)\chi_{[0,1)}(2^mt-j).
\end{equation}
We prove this claim by induction on $m$. Indeed, the identity~\eqref{eq:claim} holds for $m=0$. Let $m \ge 1$. If \eqref{eq:claim} holds for $m-1$, then
\begin{equation} \label{eq: g_(m-1) expands on the m-scale}
	\begin{aligned} 
		g_{m-1}(t)=&\sum_{j=-\infty}^\infty g(2^{1-m}j)\chi_{[0,2)}(2^m t-2j) \\ 
		=& \sum_{j=-\infty}^\infty g(2^{1-m}j)\lr{\chi_{[0,1)}(2^mt-2j)+\chi_{[0,1)}(2^mt-2j-1)} .
	\end{aligned}
\end{equation}
Note that $\alpha_{m,j} = 0$ for even $j$. By definition~\eqref{eq:gm},
	\[
		g_m(t)=g_{m-1}(t)+\sum_{j=-\infty}^\infty\lr{g(2^{-m}(2j+1))-g(2^{1-m} j)}\chi_{[0,1)}(2^mt-2j-1).
	\]
Substituting~\eqref{eq: g_(m-1) expands on the m-scale} into the above equation and merging terms containing same $\chi$'s yield the validity of~\eqref{eq:claim} for $m$. The claim is proved.

Secondly, for a fixed $t\in\mb{R}$, the index $j=\lfloor2^mt\rfloor$ satisfies $2^{-m}j\le t<2^{-m}(j+1)$. Then $g_m(t)=g(2^{-m}j)$ and
	\[
		\abs{g(t)-g_m(t)}\le\nm{g'}{L^\infty(\mb{R})}\abs{t-2^{-m}j}\le2^{-m}\nm{g'}{L^\infty(\mb{R})}.
	\]
	Therefore $g_m\to g$ in $L^\infty(\mb{R})$ when $m\to\infty$.

	Finally, for $l=0$, we have $\abs{\alpha_{0,j}}\le\nm{g}{L^\infty(\mb{R})}$. And for $l\ge 1$, we obtain
\begin{equation*}
\begin{aligned}
\abs{\alpha_{l,j}} & \le\nm{g'}{L^\infty(\mb{R})}\abs{2^{-l}j-2^{1-l}\lfloor j/2\rfloor} 
\le 2^{-l}\nm{g'}{L^\infty(\mb{R})}.
\end{aligned}
\end{equation*}
This completes the proof.
\end{proof}

In the proof of Theorem~\ref{Theorem: shallow NN Lp approximation}, we utilize the following Khintchine inequality after the symmetrization procedure.
\begin{lemma}[Khintchine inequality{~\cite{haagerup1981best}}] \label{Lemma: Khintchine ineq with optimal constant}
If $\left\{\tau_{i}\right\}_{i=1}^{m}$  are i.i.d. Rademacher random variables with  $\mathbb{P}(\tau_{i}=\pm1)=1/2$  and  $\left\{c_{i}\right\}_{i=1}^{m}\subset \mathbb{R}$, then for every  $p\ge 2$,
\begin{equation}\label{eq:Khintchine}
\left(\mb{E}\left|\sum_{i=1}^{n} c_{i} \tau_{i}\right|^{p}\right)^{1/p}  \le C_{p}\left(\sum_{i=1}^{n} c_{i}^{2}\right)^{1 / 2},
\end{equation}
where the optimal constant $C_{p} = \sqrt{2 \pi^{-1/p}} \Gamma((p+1)/2)^{1/p}$, with the Gamma function
\[
\Gamma(s){:}=\int_0^\infty t^{s-1}e^{-t}\mr{d}t, \qquad s>0.
\]
\end{lemma}

To simplify the constant $C_{p}$, we prove the following sharp bound, which captures the critical scale $\sqrt{p}$ of $C_{p}$ and takes equality when $p=2$.
\begin{lemma}\label{Lemma: Bound for Gamma function}
The optimal constant $C_{p}$ in Khintchine inequality~\eqref{eq:Khintchine} admits 
\[
C_{p} \le \sqrt{\dfrac{p}2} \quad \text{ for all } p\ge 2. 
\]
The equal sign is attained when $p = 2$.
\end{lemma}

We postpone the proof of Lemma~\ref{Lemma: Bound for Gamma function} to Appendix~\ref{Appendix Section: Bound for Gamma function}.
\vskip .05cm
\noindent
\begin{proof}[Proof of Theorem~\ref{Theorem: shallow NN Lp approximation}]
For $f\in\ms{B}^s(\Omega)$, we continue to use $f$ to denote the extension function on the entire region $\mb{R}^d$ that attains the infimum of the norm. Then $\wh{f}\in L^{1}(\mb{R}^d)$ and by the Fourier inversion theorem~\cite[Lemma 2.1]{LiaoMing:2023}, 
\[
f(x)=\int_{\mb{R}^d}\wh{f}(\xi)e^{2\pi i\xi\cdot x}\mr{d}\xi=\int_{\mb{R}^d}\abs{\wh{f}(\xi)}\cos(2\pi \xi\cdot x+\theta(\xi))\mr{d}\xi,
\]
where $\wh{f}(\xi)=\abs{\wh{f}(\xi)}e^{i\theta(\xi)}$. We apply~\eqref{eq:gm} to $\cos(2\pi \xi\cdot x+\theta(\xi))$ with $t=\xi\cdot x$, then
\begin{equation} \label{eq: expansion of f in shallow NN approximation}
f(x)=\int_{\mb{R}^d}\abs{\wh{f}(\xi)}\sum_{l=0}^\infty\sum_{j=-\infty}^\infty\alpha_{\xi,l,j}\chi_{[0,1)}(2^l\xi\cdot x-j)\mr{d}\xi ,
\end{equation}
with $\abs{\alpha_{\xi,l,j}}\le 2^{1-l}\pi$. For $x\in[0,1]^d$, note that $2^l\xi\cdot x-j$ takes values in an interval of length not greater than $2^l\abs{\xi}_1$. Therefore, the second summation is finite in the sense that
\[
f(x)=\int_{\mb{R}^d}\abs{\wh{f}(\xi)}\sum_{l=0}^\infty\sum_{j\in\mc{J}(\xi,l)}\alpha_{\xi,l,j}\chi_{[0,1)}(2^l\xi\cdot x-j)\mr{d}\xi,
\]	
with $\#\mc{J}(\xi,l)\le\lceil2^l\abs{\xi}_1\rceil\le 2^{l}(1+\abs{\xi}_1)$. 
We rewrite the above expansion of $f$ as an expectation. Define the probability measure
\begin{equation*}\label{eq:prob, shallow NN}
\mu(\mr{d}\xi,\mr{d}r){:}=\frac{1-2^{-(1+s)}}{Q} \sum_{l=0}^{\infty} 2^{-(1+s) l}\left(1+\abs{\xi}_1\right)^{-s} \abs{\wh{f}(\xi)}\mr{d}\xi\delta_l(\mr{d}r),
\end{equation*}
where $Q = \int_{\mb{R}^{d}} \left(1+\abs{\xi}_1\right)^{-s} \abs{\wh{f}(\xi)}\mr{d}\xi$ and $\delta_l$ is the Dirac measure on the integer $l$. Hence, 
\[
f(x)=\frac{Q}{1-2^{-(1+s)}} \mb{E}_{(\xi,r)\sim\mu}F(x,\xi,r)
\]
with 
\begin{equation*}\label{eq:prob}
F(x,\xi,r)=2^{(1+s)r} (1+\abs{\xi}_1)^s\sum_{j\in\mc{J}(\xi,r)}\alpha_{\xi,r,j}\chi_{[0,1)}(2^r\xi\cdot x-j).
\end{equation*}

Then, we construct an unbiased Monte Carlo approximation of $f$. Let $\Xi=\{(\xi_i,r_i)\}_{i=1}^m$ be a sequence of i.i.d. random samples from $\mu$, and 
\[
f_{\Xi}(x)=\frac{Q}{(1-2^{-(1+s)}) m}\sum_{i=1}^mF(x,\xi_i,r_i).
\]
Let $\Xi^{\prime}=\{(\xi_i^{\prime},r_i^{\prime})\}_{i=1}^m$ be an independent copy of $\Xi$. Then, by conditional expectation, independence and Jensen's inequality, 
\begin{equation*}
\begin{aligned}
\mb{E}_{\Xi}\|f-f_{\Xi}\|_{L^p(\Omega)}^{2} & \le \frac{Q^{2}}{(1-2^{-(1+s)})^{2} m^{2}} \mb{E}_{\Xi}\left\|\mb{E}_{\Xi^{\prime}}\!\left[\sum_{i=1}^m \left(F(\cdot,\xi_i,r_i) - F(\cdot,\xi_i^{\prime},r_i^{\prime})\right)\right]\right\|_{L^p(\Omega)}^{2}\\
& \le \frac{Q^{2}}{(1-2^{-(1+s)})^{2} m^{2}} \mb{E}_{\Xi,\Xi^{\prime}}\left\|\sum_{i=1}^m \left[F(\cdot,\xi_i,r_i) - F(\cdot,\xi_i^{\prime},r_i^{\prime})\right]\right\|_{L^p(\Omega)}^{2}.
\end{aligned}
\end{equation*}
Let $\tau_{i}$ be i.i.d. Rademacher random variables with $\mb{P}(\tau_{i}= \pm 1)  = 1/2$, independent of $\Xi$ and $\Xi^{\prime}$. Using standard symmetrization, we obtain
\begin{align}\label{ineq: symmetrization in Lp norm, shallow NN}
\mb{E}_{\Xi}\|f-f_{\Xi}\|_{L^p(\Omega)}^{2} &  \le \frac{Q^{2}}{(1-2^{-(1+s)})^{2} m^{2}} \mb{E}_{\Xi,\Xi^{\prime},\tau}\left\|\sum_{i=1}^m \tau_{i}\left[F(\cdot,\xi_i,r_i) - F(\cdot,\xi_i^{\prime},r_i^{\prime})\right]\right\|_{L^p(\Omega)}^{2}\nonumber\\
& \le \frac{4 Q^{2}}{(1-2^{-(1+s)})^{2} m^{2}} \mb{E}_{\Xi,\tau}\left\|\sum_{i=1}^m \tau_{i} F(\cdot,\xi_i,r_i) \right\|_{L^p(\Omega)}^{2}.
\end{align}
Conditioned on $\Xi$, it follows from Lemma~\ref{Lemma: Khintchine ineq with optimal constant} that 
\begin{equation*} 
\begin{aligned}
\mb{E}_{\tau}\left|\sum_{i=1}^m \tau_{i} F(x,\xi_i,r_i) \right|^{p} \le C_{p}^{p} \left(\sum_{i=1}^m |F(x,\xi_i,r_i)|^{2}\right)^{p/2} 
\end{aligned}
\end{equation*}
with the optimal constant $C_{p}$.
By Jensen's inequality, Fubini's theorem and the above estimate, we have
\begin{equation} \label{ineq: conditioned estimation when Lp, shallow NN}
\begin{aligned}
\mb{E}_{\tau}&\left\|\sum_{i=1}^m \tau_{i} F(\cdot,\xi_i,r_i) \right\|_{L^p(\Omega)}^{2} 
\le \left( \int_{\Omega}\mb{E}_{\tau}\left|\sum_{i=1}^m \tau_{i} F(x,\xi_i,r_i) \right|^{p} \mr{d}x\right)^{2/p}\\
&\le C_{p}^{2} \left\|\sum_{i=1}^m |F(\cdot,\xi_i,r_i)|^{2}\right\|_{L^{p/2}(\Omega)}\\
&\le C_{p}^{2} \sum_{i=1}^m\left\| F(\cdot,\xi_i,r_i)\right\|_{L^{p}(\Omega)}^{2} ,\\
\end{aligned}
\end{equation}
where we have used the triangle inequality in the last step. Note that the support of each term in $F$ is disjoint, hence
\[
\|F(\cdot,\xi,r)\|_{L^p(\Omega)} \le  \|F(\cdot,\xi,r)\|_{L^\infty(\Omega)}\le 2^{1+sr}\pi (1+\abs{\xi}_1)^s,
\]
and 
\begin{equation*} 
\begin{aligned}
\mb{E}_{\Xi}\|F(\cdot,\xi_{i},r_{i})\|_{L^p(\Omega)}^{2}  
& \le \frac{4\pi^{2}\left(1-2^{-(1+s)}\right)}{Q} \sum_{l=0}^{\infty} 2^{-(1-s)l}\int_{\mb{R}^d} (1+\abs{\xi}_1)^{s}\abs{\wh{f}(\xi)} \mr{d}\xi \\
& \le \frac{4\pi^{2}\left(1-2^{-(1+s)}\right)}{\left(1-2^{-(1-s)}\right)Q}  \|f\|_{\ms{B}^s(\Omega)} .
\end{aligned}
\end{equation*}
A combination of (\ref{ineq: symmetrization in Lp norm, shallow NN}), (\ref{ineq: conditioned estimation when Lp, shallow NN}) and the above estimate yields 
\begin{equation*} 
\begin{aligned}
\mb{E}_{\Xi}\|f-f_{\Xi}\|_{L^p(\Omega)}^{2}  
& \le \frac{16\pi^{2} C_{p}^{2} Q \|f\|_{\ms{B}^s(\Omega)}}{\left(1-2^{-(1+s)}\right) \left(1-2^{-(1-s)}\right)m}   .
\end{aligned}
\end{equation*}
By Markov's inequality, with probability at least $(1+\varepsilon)/(2+\varepsilon)$, for some $\varepsilon>0$ to be chosen later on, we have
\begin{equation}\label{eq: Lp error, shallow NN}
\|f-f_{\Xi}\|_{L^p(\Omega)}^{2} \le  \frac{16 (2+\varepsilon) \pi^{2} C_{p}^{2} Q \|f\|_{\ms{B}^s(\Omega)}}{\left(1-2^{-(1+s)}\right) \left(1-2^{-(1-s)}\right)m} .
\end{equation}

Next, we control the number $N_{\Xi}$ of Heaviside function used to represent $f_{\Xi}$. Note that $\chi_{[0,1)}$ can be represented by the Heaviside function
$$\chi_{[0,1)}(t)=\chi_{[0,\infty)}(t)-\chi_{[0,\infty)}(t-1),$$
it remains to calculate the total number of $\chi_{[0,1)}$ in $f_{\Xi}$. It follows from $\#\mc{J}(\xi,l)\le 2^{l}(1+\abs{\xi}_1)$ that $N_{\Xi}\le\sum_{i=1}^m 2^{1+r_i}(1+\abs{\xi_i}_1)$ and
\begin{equation*}
\begin{aligned}
\mb{E}N_{\Xi}^{2s} & \le 2^{2s}\mb{E}_{\Xi}\sum_{i=1}^m2^{2sr_i}(1+\abs{\xi_i}_1)^{2s} \le \frac{2^{2s}\left(1-2^{-(1+s)}\right)m}{\left(1-2^{-(1-s)}\right)Q}  \|f\|_{\ms{B}^s(\Omega)}  .
\end{aligned}
\end{equation*}
By Markov's inequality, with probability at least $(1+\varepsilon)/(2+\varepsilon)$, we obtain
\begin{equation}\label{eq: number of Chi, shallow NN}
\dfrac{Q}{\left(1-2^{-(1+s)}\right) m} \le \frac{(2+\varepsilon) 2^{2s} \|f\|_{\ms{B}^s(\Omega)}}{\left(1-2^{-(1-s)}\right) N_{\Xi}^{2s}}  .
\end{equation}
Combining~\eqref{eq: Lp error, shallow NN} and~\eqref{eq: number of Chi, shallow NN}, with probability at least $\varepsilon/(2+\varepsilon)$, we obtain
\begin{equation*}
\begin{aligned}
\|f-f_{\Xi}\|_{L^p(\Omega)} 
& \le  \frac{2^{5/2} (2+\varepsilon) \pi C_{p} \|f\|_{\ms{B}^s(\Omega)}}{\left(1-2^{-1/2}\right)N_{\Xi}^{s}} .
\end{aligned}
\end{equation*}
Choosing $\varepsilon = 10^{-3}$ and using Lemma~\ref{Lemma: Bound for Gamma function}, we get~\eqref{ineq: shallow NN Lp approximation bound}.

Finally, since the Heaviside function may be well approximated by the sigmoidal function according to Lemma~\ref{lema:sigmoidal}, we replace $\chi_{[0,1)}(2^{r_{i}}\xi_{i}\cdot x-j)$ in $f_{\Xi}$ by \[
\sigma(R_i(2^{r_{i}}\xi_{i}\cdot x-j))-\sigma(R_i(2^{r_{i}}\xi_{i}\cdot x-j-1))
\]
with certain sufficiently large $R_i$, which completes the proof.
\end{proof}

As far as we know, the $L^p$-estimate in Theorem~\ref{Theorem: shallow NN Lp approximation} is currently the sharpest result available in the literature. \cite[Theorem 3]{Makovoz1996Random} deals with general $p$ while requires $f\in\ms{B}^1(\Omega)$. Specializing to the case $p=2$, \cite[Theorem 3]{Siegel:2022} obtains the approximation rate $\mathcal{O}(N^{-1/2})$ for $f \in \ms{B}^{1/2}(\Omega)$, but the constant in their estimate grows exponentially with $d$.
If, instead, one samples $\chi_{[0,1)}(2^l\xi\cdot x-j)$ in the expansion~\eqref{eq: expansion of f in shallow NN approximation} directly, as in \cite{Siegel:2022} for the $L^{2}$-approximation, then the rate $\mc{O}(N^{-1/2})$ may only be attained for $f\in\ms{B}^{1-1/p}(\Omega)$. This is due to their use of the key estimate
\[
\|\chi_{[0,1)}(w\cdot x +b)\|_{L^{p}(\Omega)} \le C(d)|w|^{-1/p}.
\]

Moreover, our results provide error bounds that depend solely on the seminorm $\upsilon_{f,s}$, which yields improvements when the target function satisfies $\upsilon_{f,s} \ll \upsilon_{f,0}$.  We refer to~\cite[eq. (2.5)]{LiaoMing:2023} for an example illustrating this scenario.
\begin{corollary} \label{Corollary: approximation bounds with spectral semi-norm shallow NN}
Let $p\ge 1$ and $f\in\ms{B}^s(\Omega)$ with $0<s\le 1/2$. For any  $N\in \mathbb{N}_{+}$, there exists a shallow sigmoidal network with $N$ units such that 
\begin{equation*}
\|f-f_N\|_{L^p(\Omega)} \le C\sqrt{p}\,\upsilon_{f,s} N^{-s},
\end{equation*}
where $C$ is a universal constant.
\end{corollary}

We recover~\cite[Theorem 3.5]{LiaoMing:2023} when $p=2$.
\begin{proof}
We only consider the case $p\ge 2$, because for $1\le p<2$, 
\[
\|f-f_{N}\|_{L^p(\Omega)} \le \|f-f_{N}\|_{L^2(\Omega)}.
\]

The main idea of the proof is that we decompose $f$ into low frequency part $f_1$ and high frequency part $f_2$, and  apply Lemma~\ref{Lemma: shallow sigmoidal Lp app B^1 by Makovoz1996} and Theorem~\ref{Theorem: shallow NN Lp approximation} to $f_1$ and $f_2$, respectively. 

We decompose $f=f_1+f_2$ with
\begin{equation} \label{eq: decompose f into low and high frequency parts}
f_1(x)=\operatorname{Re}\int_{\abs{\xi}_1<1}\wh{f}(\xi)e^{2\pi i\xi\cdot x}\mr{d}\xi,\quad f_2(x)=\operatorname{Re}\int_{\abs{\xi}_1\ge 1}\wh{f}(\xi)e^{2\pi i\xi\cdot x}\mr{d}\xi.
\end{equation}
It is straightforward to verify
\[\wh{f}_1(\xi) =\frac{1}{2}[\wh{f}(\xi)+\wh{f}(-\xi)^{*}]\chi_{[0,1)}(\abs{\xi}_1), \qquad\wh{f}_2(\xi)=\frac{1}{2}[\wh{f}(\xi)+\wh{f}(-\xi)^{*}]\chi_{[1,\infty)}(\abs{\xi}_1),
\]
where $z^{*}$ means the complex conjugate of $z$. 
Let $N_{1} = \lceil N/2\rceil$ and $N_{2} = \lfloor N/2\rfloor$.
It follows from $\upsilon_{f_1,1}\le\upsilon_{f_1,s}\le\upsilon_{f,s}$ that $f_{1}\in \ms{B}^1(\Omega)$, and by Lemma~\ref{Lemma: shallow sigmoidal Lp app B^1 by Makovoz1996}, there exists a shallow sigmoidal network $f_{1,N_{1}}$ with $N_{1}$ units such that 
\begin{equation} \label{ineq: approximate low frequency part, shallow NN}
\|f_{1}-f_{1,N_{1}}\|_{L^p(\Omega)} \le C\sqrt{p} \upsilon_{f_1,1}N_1^{-1/2} \le C\sqrt{p}\upsilon_{f,s} N^{-s}.
\end{equation}

Similarly, if $N\ge2$, it follows from $\|f_{2}\|_{\ms{B}^s(\Omega)}\le 2^{s}\upsilon_{f_2,s}\le\sqrt{2}\upsilon_{f,s}$ and Theorem~\ref{Theorem: shallow NN Lp approximation} that $f_{2}\in \ms{B}^s(\Omega)$ and there exists a shallow sigmoidal network with $N_{2}$ units such that 
\begin{equation} \label{ineq: approximate high frequency part, N>1, shallow NN}
\|f_{2}-f_{2,N_{2}}\|_{L^p(\Omega)} \le \frac{C \sqrt{p} \|f_{2}\|_{\ms{B}^s(\Omega)}}{N_{2}^{s}} \le \frac{C \sqrt{p} \upsilon_{f,s}}{N^{s}} , 
\end{equation}
where we used $N_{2}\ge N/3$ in the second inequality. 

For $N = 1$, we simply choose $f_{2,N_{2}} = 0$ and 
\begin{equation} \label{ineq: approximate high frequency part, N=1, shallow NN}
\|f_{2}-f_{2,N_{2}}\|_{L^p(\Omega)} \le \|f_{2}\|_{L^{\infty}(\Omega)} \le \|f_{2}\|_{\ms{B}^s(\Omega)}\le \sqrt{2} \upsilon_{f,s} . 
\end{equation}

Now, we take $f_{N} = f_{1,N_{1}} + f_{2,N_{2}}$. Then, $f_{N}$ contains at most $N$ units. We conclude from (\ref{ineq: approximate low frequency part, shallow NN}), (\ref{ineq: approximate high frequency part, N>1, shallow NN}) and (\ref{ineq: approximate high frequency part, N=1, shallow NN}) that
\begin{equation*} 
\begin{aligned}
\|f-f_{N}\|_{L^p(\Omega)} & \le \|f_{1}-f_{1,N_{1}}\|_{L^p(\Omega)} + \|f_{2}-f_{2,N_{2}}\|_{L^p(\Omega)} \\
& \le C\sqrt{p} \upsilon_{f,s} N^{-s},
\end{aligned}
\end{equation*}
where $C$ is a universal constant. This completes the proof.
\end{proof}

%% file: Section/Lp_approximation.tex
\section{Deep ReLU networks approximation}\label{Section: deep neural network approximation}
In this part,  we extend the results in the previous section to deep ReLU networks. 
Besides $L^p$-approximation, we also prove the $L^\infty$-approximation bounds.  
We define an $(L,N)$-network as a ReLU network with $L$ hidden layers and at most $N$ units per layer 
\begin{equation*}
f_{N}(x) = W_{L}\operatorname{ReLU}(W_{L-1}\cdot)\circ\cdots\circ\operatorname{ReLU}(W_{1}\cdot)\circ \operatorname{ReLU}(W_{0}x+b_{0}) + b_{L},
\end{equation*}
where $W_{0}\in \mb{R}^{N\times d}$, $W_{1},\ldots,W_{L-1}\in \mb{R}^{N\times N}$, $W_{L}\in \mb{R}^{1\times N}$ and $b_{0}\in \mb{R}^{N}$, $b_{L}\in \mb{R}$. With this definition, the shallow ReLU neural network is a $(1,N)$-network.
\subsection{\texorpdfstring{$L^{p}$}{Lp}-approximation} 
The fundamental work on DNN approximation of the spectral Barron functions is~\cite[Theorem 1]{BreslerNagaraj:2020}, where the authors have proven $L^{2}$-approximation rate of $\mathcal{O}(N^{-sL/2})$. We have improved the approximation rate to $\mathcal{O}(N^{-sL})$ in~\cite[Theorem 3.5]{LiaoMing:2023}.
As summarized in Theorem~\ref{Theorem: Lp approximation by deep NN} and Corollary~\ref{Corollary: approximation bounds with spectral semi-norm DNN}, we shall extend this result to $L^{p}$-approximation bounds, demonstrating that the rate $\mathcal{O}(N^{-sL})$ can be preserved in $L^{p}(\Omega)$ for all $p\ge 1$\footnote{Note that the $L^{p}(\Omega)$ norm is controlled by $L^{2}(\Omega)$ norm for $1\le p\le 2$.}.
\begin{theorem} \label{Theorem: Lp approximation by deep NN}
Let $p \ge 2$, $L\in \mathbb{N}_{+}$ and $f\in\ms{B}^s(\Omega)$ with $0<sL\le 1/2$. For any positive integer $N$, there exists an $(L,N)$-network $f_{L,N}$ such that
\begin{equation}\label{eq: Lp approximation error by deep NN}
\|f-f_{L,N}\|_{L^p(\Omega)} \le 13 \sqrt{p} \nm{f}{\ms{B}^s(\Omega)} N^{-sL}. 
\end{equation}
\end{theorem}

Inspired by~\cite{BreslerNagaraj:2020}, our proof is based on an integral representation that is more suitable for handling DNN.  
The integral representation was introduced by authors in~\cite{BreslerNagaraj:2020} for DNN, offering an alternative to the expansion in~\eqref{eq: expansion of f in shallow NN approximation} that is commonly used for SNN. Here, we adopt this integral representation, and restate it below with more explicit and accessible notation for the reader’s convenience. A similar representation has been exploited by two of the authors in~\cite{LiaoMing:2023}.

First, we make some preparations for the rest work.  For any function $g$ defined on $[0,1]$ and it is symmetric about $x=1/2$, We use the notation $g_{,n}$ to denote the function $g$ in the $[0,1]$ interval of the period repeated $n$ times, i.e.,
\begin{equation}\label{eq:gn}
	g_{,n}(t)=g(nt-j),\quad j=0,\dots,n-1,\quad 0\le nt-j\le 1.
\end{equation}
Define a hat function
\[
	\beta(t)=\text{ReLU}(2t)-2\text{ReLU}(2t-1)+\text{ReLU}(2t-2)=\begin{cases}
		2t,&0\le t\le 1/2,\\
		2-2t,&1/2\le t\le 1,\\
		0,&\text{otherwise}.
	\end{cases}
\]
By definition~\eqref{eq:gn}, $\beta_{,n}$ represents a triangle function with $n$ peaks and can be represented by $3n$ ReLUs:
\[
\beta_{,n}(t)=\sum_{j=0}^{n-1}\beta(nt-j),\quad 0\le t\le 1.
\]
The following lemma describes the effect of the composition with $\beta_{,n}$. A geometrical explanation may be founded in~\cite[Figure 3]{Bolcskei:2021}. An interesting example is $g=\cos$ and $\cos(2\pi n_2\beta_{,n_1}(t))=\cos(4\pi n_1n_2t)$ when $t\in [0, 1]$. 
\begin{lemma}[{\cite[Lemma 3.5]{LiaoMing:2023}}]\label{lema:comp}
	Let $g$ be a function defined on $[0,1]$ and symmetric about $x=1/2$, then $g_{,n_2}\circ\beta_{,n_1}=g_{,2n_1n_2}$ on $[0,1]$.
\end{lemma}

For $r\in(0,1)$, we define 
\begin{equation*}
\begin{aligned}
\alpha(t,r) & =\operatorname{ReLU}(t)-\operatorname{ReLU}(t-r/2)-\operatorname{ReLU}(t-(1-r)/2)+\operatorname{ReLU}(t-1/2) \\
& = \begin{cases}
t, & 0 \le t \le\min(r,1-r)/2,\\
\min(r,1-r)/2, & \min(r,1-r)/2\le t\le \max(r,1-r)/2,\\
1/2 - t, & \max(r,1-r)/2\le t\le 1/2,\\
0, & \text{otherwise},\\
\end{cases}
\end{aligned}
\end{equation*}
then supp$(\alpha(\cdot,r))\subset[0,1/2]$ and $\alpha(t,r)$ is symmetric about $t=1/4$. Define $$\gamma(t,r)=\alpha(t+1/4,r)-\alpha(t-1/4,r)+\alpha(t-3/4,r).$$
Then $\gamma(t,r)$ is symmetric about $t=1/2$ because
\begin{align*}
\gamma(1-t,r)&=\alpha(5/4-t,r)-\alpha(3/4-t,r)+\alpha(1/4-t,r)\\
&=\alpha(t-3/4,r)-\alpha(t-1/4,r)+\alpha(t+1/4,r)=\gamma(t,r).
\end{align*}
By definition~\eqref{eq:gn}, $\gamma_{,n}(\cdot,r)$ is well defined on $[0,1]$ and
\begin{align*}
\gamma_{,n}(t,r)=&\begin{cases}
\alpha(nt-j+1/4,r), & 0\le nt-j \le 1/4, \\
-\alpha(nt-j-1/4,r), & 1/4\le nt-j\le 3/4,\\
\alpha(nt-j-3/4,r), & 3/4\le nt-j\le 1,
\end{cases}&& j=0,\dots,n-1,\\
=&\begin{cases}
\alpha(nt-j+1/4,r), & -1/4\le nt-j\le 1/4,\\
-\alpha(nt-j-1/4,r), & 1/4\le nt-j\le 3/4,\\
\end{cases} && j=0,\dots,n.
\end{align*}
In fact, $\gamma_{,n}(\cdot,r)$ on $[0,1]$ can be represented by $4n$ ReLU functions plus a constant 
\begin{equation*}
\begin{aligned}
\gamma_{,n}(t,r) & =\sum_{j = 0}^{n}\alpha(nt-j+1/4,r) - \sum_{j = 0}^{n-1}\alpha(nt-j-1/4,r) \\
& = \frac{1}{2}\min(r,1-r) + \sum_{j = 0}^{n-1} \left[-\operatorname{ReLU}\left(nt-j+\frac{1}{4}-\frac{\max(r,1-r)}{2}\right) \right. \\
& \qquad  + \operatorname{ReLU}\left(nt-j-\frac{1}{4}-\frac{r}{2}\right) + \operatorname{ReLU}\left(nt-j-\frac{3}{4}+\frac{r}{2}\right)  \\
& \qquad - \left. \operatorname{ReLU}\left(nt-j-\frac{3}{4}-\frac{\min(r,1-r)}{2}\right)\right] .
\end{aligned}
\end{equation*}

A direct consequence of the above construction is~\cite[Lemma 3]{BreslerNagaraj:2020}. 
%
\begin{lemma}\label{Lemma: Integral representation of cos by gamma}
For $t\in[0,1]$, there holds
\begin{equation}\label{eq: Integral representation of cos by gamma,n(t,r)}
	\pi^{2}\int_0^1\sin(\pi r)\gamma_{,n}(t,r)\mr{d}r=\cos(2\pi nt).
\end{equation}
\end{lemma}

Now we are ready to prove the $L^{p}$-approximation bound for DNN.

\begin{proof}[Proof of Theorem~\ref{Theorem: Lp approximation by deep NN}]
For $f\in\ms{B}^s(\Omega)$, we continue to use $f$ to denote the extension function on the entire region $\mb{R}^d$ that attains the infimum of the norm. Then $\wh{f}\in L^{1}(\mb{R}^d)$ and by Fourier inversion theorem~\cite[Lemma 2.1]{LiaoMing:2023}, 
\[
f(x)=\int_{\mb{R}^d}\wh{f}(\xi)e^{2\pi i\xi\cdot x}\mr{d}\xi=\int_{\mb{R}^d}\abs{\wh{f}(\xi)}\cos(2\pi(\xi\cdot x+\theta(\xi)))\mr{d}\xi,
\]
with proper choice $\theta(\xi)$ such that $0\le\xi\cdot x+\theta(\xi)\le1+\abs{\xi}_1$. For fixed $\xi$, choose $n_\xi=2^{L-1}\lceil(1+\abs{\xi}_1)^{1/L}\rceil^L$ and $t_\xi(x)=(\xi\cdot x+\theta(\xi))/n_\xi$, then $0\le t_\xi(x)\le 1$ and by Lemma~\ref{Lemma: Integral representation of cos by gamma},
\[
f(x) =\int_{\mb{R}^d}\abs{\wh{f}(\xi)}\cos(2\pi n_\xi t_\xi(x))\mr{d}\xi=\pi^{2}\int_{\mb{R}^d} \int_0^1 \abs{\wh{f}(\xi)}\sin(\pi r)\gamma_{,n_\xi}(t_\xi(x),r)\mr{d}r \mr{d}\xi.
\]
Define the probability measure
\begin{equation*}\label{eq:probability distribution of sampling}
\mu(\mr{d}\xi,\mr{d}r){:}=\frac{\pi}{2Q}(1+\abs{\xi}_1)^{-s}\abs{\wh{f}(\xi)} \sin(\pi r)\chi_{(0,1)}(r)\mr{d}\xi\mr{d}r, 
\end{equation*}
where $Q = \int_{\mb{R}^{d}} \left(1+\abs{\xi}_1\right)^{-s} \abs{\wh{f}(\xi)}\mr{d}\xi$.  Therefore,
\begin{equation} \label{eq: expectation representation of f for DNN approximation}
f(x)= 2\pi Q\mb{E}_{(\xi,r)\sim\mu}F(x,\xi,r)
\end{equation}
with 
\[
F(x,\xi,r)= (1+\abs{\xi}_1)^{s}\gamma_{,n_\xi}(t_\xi(x),r).
\]
Let $\Xi = \{(\xi_i,r_i)\}_{i=1}^m$ be an i.i.d. sequence of random samples from $\mu$, and $$f_{\Xi}(x)=\frac{2\pi Q}{m}\sum_{i=1}^m F(x,\xi_i,r_i).$$ 
Let $\Xi^{\prime}= \{(\xi_i^{\prime},r_i^{\prime})\}_{i=1}^m$ be an independent copy of $\Xi$. Then, as in the proof of Theorem~\ref{Theorem: shallow NN Lp approximation}, by conditional expectation, independence and Jensen's inequality, 
\begin{equation*}
\begin{aligned}
\mb{E}_{\Xi}\|f-f_{\Xi}\|_{L^p(\Omega)}^{2} 
& \le \frac{4\pi^{2} Q^{2}}{m^{2}} \mb{E}_{\Xi,\Xi^{\prime}}\left\|\sum_{i=1}^m \left[F(\cdot,\xi_i,r_i) - F(\cdot,\xi_i^{\prime},r_i^{\prime})\right]\right\|_{L^p(\Omega)}^{2}.
\end{aligned}
\end{equation*}
By the same argument to obtain (\ref{ineq: symmetrization in Lp norm, shallow NN}) and (\ref{ineq: conditioned estimation when Lp, shallow NN}), we obtain
\begin{equation} \label{ineq: symmetrization and conditioned estimation when Lp, DNN}
\begin{aligned}
\mb{E}_{\Xi}\|f-f_{\Xi}\|_{L^p(\Omega)}^{2} & \le \frac{16\pi^{2} Q^{2}}{m^{2}}C_{p}^{2}   \sum_{i=1}^m\mb{E}_{\Xi}\left\| F(\cdot,\xi_i,r_i)\right\|_{L^{p}(\Omega)}^{2}.
\end{aligned}
\end{equation}
Note that $\|F(\cdot,\xi,r)\|_{L^p(\Omega)} \le  \|F(\cdot,\xi,r)\|_{L^\infty(\Omega)}\le (1+\abs{\xi}_1)^{s}/4$
and 
\[
\mb{E}_{(\xi,r)\sim \mu}\|F(\cdot,\xi,r)\|_{L^p(\Omega)}^{2}\le \frac{\pi}{32Q} \int_{\mb{R}^d} \int_0^1(1+\abs{\xi}_1)^{s}\abs{\wh{f}(\xi)} \sin(\pi r)\mr{d}r\mr{d}\xi\le \frac{\|f\|_{\ms{B}^{s}(\Omega)}}{16Q}.
\]
A combination of (\ref{ineq: symmetrization and conditioned estimation when Lp, DNN}) and the above estimate yields that
\[
\mb{E}_{\Xi}\|f-f_{\Xi}\|_{L^p(\Omega)}^{2}  
\le \frac{\pi^{2}}{m} C_{p}^{2} Q   \|f\|_{\ms{B}^{s}(\Omega)} .
\]
By Markov's inequality, with probability at least $(1+\varepsilon)/(2+\varepsilon)$, for some $\varepsilon>0$ to be chosen later on, we have
\begin{equation}\label{eq: Lp error}
\|f-f_{\Xi}\|_{L^p(\Omega)}^{2} \le \frac{\pi^{2}(2+\varepsilon)}{m} C_{p}^{2} Q   \|f\|_{\ms{B}^{s}(\Omega)} .
\end{equation}
	
It remains to calculate the number of units in each layer. For each $\gamma_{,n_\xi}(t_\xi(x),r)$, choose $n_1=\dots=n_L=\lceil(1+\abs{\xi}_1)^{1/L}\rceil$, then $n_\xi=2^{L-1}n_1\dots n_L$, and by Lemma~\ref{lema:comp}, $\gamma_{,n_\xi}(\cdot,r)=\gamma_{,n_L}(\cdot,r)\circ\beta_{,n_{L-1}}\circ\dots\circ\beta_{,n_1}$ on $[0,1]$. We need at most
$$\max\{3n_1,\dots,3n_{L-1},4n_L\} = 4\lceil(1+\abs{\xi}_1)^{1/L}\rceil \le 8(1+\abs{\xi}_1)^{1/L}$$
units in each layer to represent $\gamma_{,n_\xi}(t_\xi(x),r)$. 
Denote $N_{\Xi}$ the total number of units in each layer, then $N_{\Xi}\le 8\sum_{i=1}^m(1+\abs{\xi_i}_{1})^{1/L}$ and 
\[
\mb{E}_{\Xi}N_{\Xi}^{2sL}\le 8\mb{E}_\Xi\sum_{i=1}^m(1+\abs{\xi_i}_{1})^{2s}\le\frac{8m}{Q}\|f\|_{\ms{B}^{s}(\Omega)}.
\]
Again, by Markov inequality, with probability at least $(1+\varepsilon)/(2+\varepsilon)$, we obtain
\begin{equation}\label{eq: width bound}
\dfrac{Q}{m}\le 8 (2+\varepsilon) \|f\|_{\ms{B}^{s}(\Omega)} N_{\Xi}^{-2sL}. 
\end{equation}

Combining~\eqref{eq: Lp error} and~\eqref{eq: width bound}, with probability at least $\varepsilon/(2+\varepsilon)$, there exists an $(L,N_\Xi)$-network $f_{\Xi}$ such that
\[
\|f-f_{\Xi}\|_{L^p(\Omega)} \le 2\sqrt{2}(2+\varepsilon)\pi C_{p} \|f\|_{\ms{B}^{s}(\Omega)} N_\Xi^{-sL} . 
\]

Choosing $\varepsilon = 1/100$ and using Lemma~\ref{Lemma: Bound for Gamma function} again, we complete the proof. 
\end{proof}

Similar to Corollary~\ref{Corollary: approximation bounds with spectral semi-norm shallow NN}, we may obtain error bounds with merely the seminorm $\upsilon_{f,s}$ in the right hand side of the approximation bound.
\begin{corollary} \label{Corollary: approximation bounds with spectral semi-norm DNN}
Let $p \ge 1$, $L\in \mathbb{N}_{+}$ and $f\in\ms{B}^s(\Omega)$ with $0<sL\le 1/2$. For any positive integer $N$, there exists an $(L,N)$-network $f_{L,N}$ such that
\begin{equation*}
\|f-f_{L,N}\|_{L^p(\Omega)} \le C \sqrt{p} \upsilon_{f,s} N^{-sL},
\end{equation*}
where $C$ is a universal constant. 
\end{corollary}

\begin{proof}
The proof follows a similar approach with that leads to Corollary~\ref{Corollary: approximation bounds with spectral semi-norm shallow NN}. Recall the decomposition in~\eqref{eq: decompose f into low and high frequency parts}. As before, we approximate the low-frequency component using Lemma~\ref{Lemma: shallow sigmoidal Lp app B^1 by Makovoz1996}, and the high-frequency component using Theorem~\ref{Theorem: Lp approximation by deep NN}. This is justified by the fact that a sigmoidal function can be represented by two ReLU units, as $\sigma(t) = \operatorname{ReLU}(t) - \operatorname{ReLU}(t-1)$, and shallow networks can be interpreted as a special case of deep networks since the identity function can be expressed as $t = \operatorname{ReLU}(t) - \operatorname{ReLU}(-t)$. The details are omitted for brevity.
\end{proof}

%% file: Section/Linf_approximation.tex
\subsection{Uniform approximation}
In this part, we control the $L^{\infty}$-approximation error. The main result, Theorem~\ref{Theorem: Linf approximation by deep NN}, which achieves the Monte Carlo rate under the minimal smoothness assumption known to date. We firstly recall
\begin{lemma}[{\cite[Proposition 2.2]{Caragea:2023}}]\label{lemma:Linf}
Let $f\in\ms{B}^1(\Omega)$. For any $N\in \mathbb{N}_{+}$, there exists a $(1,N)$-network $f_N$ such that
\[
\nm{f-f_N}{L^\infty(\Omega)}\le C\sqrt{d}\,\upsilon_{f,1}N^{-1/2},
\]
where $C$ is a universal constant.
\end{lemma}

We reduce the smoothness assumption to $f\in\ms{B}^{1/2}(\Omega)$, while achieve the same SNN approximation rate. The result below is fully consistent with Theorem~\ref{Theorem: Lp approximation by deep NN} up to a logarithmic factor.
\begin{theorem} \label{Theorem: Linf approximation by deep NN}
Let $L\in \mathbb{N}_{+}$ and $f\in\ms{B}^s(\Omega)$ with $0<sL\le 1/2$. For any positive integer $N$, there exists an $(L,N)$-network $f_{L,N}$ such that
\begin{equation}\label{eq: Linf approximation error by deep NN}
\|f-f_{L,N}\|_{L^{\infty}(\Omega)} \le\frac{45 \|f\|_{\ms{B}^s(\Omega)}}{N^{sL}} \sqrt{1+dL\ln N}.
\end{equation}
\end{theorem}

To prove Theorem~\ref{Theorem: Linf approximation by deep NN}, we employ the well-known tool of Rademacher complexity. We recall its definition firstly.
\begin{definition}
Define for a set of random variables  $Z = \left\{Z_{i}\right\}_{i=1}^{m}$  independently distributed according to  $P$  and a function class  $\mathcal{G}$  the empirical Rademacher complexity 
$$\widehat{\mathcal{R}}_{m}(\mathcal{G},Z):=\mathbb{E}_{\tau}\left[\sup _{g \in \mathcal{G}}\left.\frac{1}{m} \sum_{i=1}^{m} \tau_{i} g\left(Z_{i}\right)\right|  Z_{1}, \cdots, Z_{m}\right],$$
where the expectation $\mathbb{E}_{\tau}$  is taken with respect to the independent uniform Bernoulli sequence  $\left\{\tau_{i}\right\}_{i=1}^{m}$  with  $\tau_{i} \in\{ \pm 1\}$. The Rademacher complexity of  $\mathcal{G}$  is defined by \[
\mathcal{R}_{m}(\mathcal{G}) = \mathbb{ E}_{Z\sim P^{m}}\left[\widehat{\mathcal{R}}_{m}(\mathcal{G})\right].
\]
\end{definition}

Let  $(E, \rho)$  be a metric space with metric  $\rho $. A  $\delta $-cover of a set  $A \subset E$  with respect to  $\rho$  is a collection of points  $\left\{z_{1}, \cdots, z_{n}\right\} \subset A$  such that for every  $z \in A $, there exists  $i \in\{1, \cdots, n\}$  such that  $\rho\left(z, z_{i}\right) \leq \delta $. The  $\delta $-covering number  $\mathcal{N}(\delta, A, \rho)$ is the cardinality of the smallest  $\delta$-cover of the set  $A$  with respect to $\rho$. Equivalently, the  $\delta $-covering number  $\mathcal{N}(\delta, A, \rho)$  is the minimal number of balls  $B_{\rho}(z, \delta)$  of radius  $\delta$  needed to cover $A$. $\ln \mathcal{N}(\delta, A, \rho)$ is often referred to as the metric entropy of $A$.

The following lemma, which is a consequence of Massart’s Lemma, will be used to bound the Rademacher complexity in terms of the metric entropy.
\begin{lemma}[{\cite[Proposition 5.2]{Patrick2020Lecture}}] \label{Lemma: Metric entropy Bound for empirical Rademacher complexity}
Let  $\mathcal{G}$  be a class of real functions,  $Z = \left\{Z_{i}\right\}_{i = 1}^{m}$  be a random $i.i.d.$ sample and the empirical measure $P_{m} = m^{-1} \sum_{i=1}^{m} \delta_{Z_{i}} $.  For $p \ge 1$, denote 
$$ \|g\|_{L^{p}(P_{m})} := \left(\frac{1}{m} \sum_{i=1}^{m} \left|g(Z_{i})\right|^{p}\right)^{1/p}.$$
Assuming $\sup _{g \in \mathcal{G}} \|g\|_{L^{2}(P_{m})} \leq c$, then 
$$\widehat{\mathcal{R}}_{m}(\mathcal{G},Z) \leq \inf _{\delta > 0}\left( \delta + \frac{\sqrt{2} c}{\sqrt{m}} \sqrt{\ln \mathcal{N}\left(\delta, \mathcal{G}, \|\cdot\|_{L^{1}(P_{m})}\right)} \right).$$    
\end{lemma} 

\begin{proof}[Proof of Theorem~\ref{Theorem: Linf approximation by deep NN}]
Still we use $f$ to denote the extension function on the entire region $\mb{R}^d$ that attains the infimum of the norm. If $N^{2sL} \le 2^{10} \|f\|_{\ms{B}^s(\Omega)} /Q$, where $Q=\int_{\mb{R}^d}(1+\abs{\xi}_1)^{-s}\abs{\widehat{f}(\xi)}\mathrm{d}\xi$, then it follows from Hausdorff-Young inequality and Cauchy's inequality that \[
\|f\|_{L^{\infty}(\Omega)}^{2} \le \|f\|_{\ms{B}^{0}(\Omega)}^{2} \le Q\|f\|_{\ms{B}^s(\Omega)},
\]
which immediately implies 
\[
\|f\|_{L^{\infty}(\Omega)} \le 32\|f\|_{\ms{B}^s(\Omega)} N^{-sL}.
\]
Therefore, the choice of $f_N = 0$ suffices for our purpose. 

It remains to prove the approximation error bound for $\|f\|_{\ms{B}^s(\Omega)} /Q < 2^{-10}N^{2sL}$.
For some $\tilde{c}>0$ to be chosen later on, we set $m = \lceil \tilde{c} QN^{2sL}/\|f\|_{\ms{B}^s(\Omega)}\rceil$. Hence,
\begin{equation} \label{ineq: upper bound for m}
\begin{aligned}
\frac{m \|f\|_{\ms{B}^s(\Omega)}}{Q}\le\tilde{c}N^{2sL}+\frac{\|f\|_{\ms{B}^s(\Omega)}}{Q} \le (\tilde{c} + 2^{-10}) N^{2sL}.
\end{aligned}
\end{equation}
Recall that we have derived the expectation representation~\eqref{eq: expectation representation of f for DNN approximation} for $f$. 
We draw random samples $\Xi = \{(\xi_i,r_i)\}_{i=1}^m$ and construct unbiased estimate $f_{\Xi}$ as in the proof of Theorem~\ref{Theorem: Lp approximation by deep NN}.
For $x \in \Omega$, define measurable function $g_{x}(\xi, r): \mb{R}^{d}\times[0,1] \to \mb{R}$ by $g_{x}(\xi, r) = F(x,\xi,r).$ Denote the function sets $G=\left\{g_{x} : x \in \Omega\right\}$ and $\bar{G} = G \cup (-G).$ Then, by standard symmetrization,
\begin{equation} \label{ineq: symmetrization}
\begin{aligned}
\mb{E}_{\Xi}\|f-f_{\Xi}\|_{L^{\infty}(\Omega)} 
& = 2\pi Q \mb{E}_{\Xi} \sup_{x \in \Omega} \left|\dfrac1m\sum_{i=1}^m g_{x}(\xi_{i}, r_{i}) - \mb{E}_{(\xi,r)\sim\mu}g_{x}(\xi, r)\right| \\
& \leq \frac{4\pi Q}{m} \mb{E}_{\Xi,\tau} \sup_{g \in G} \left|\sum_{i=1}^m \tau_{i} g(\xi_{i}, r_{i})\right|\le 4\pi Q \mb{E}_{\Xi} \widehat{\mathcal{R}}_{m}(\bar{G},\Xi), 
\end{aligned}
\end{equation}
where $\tau_{i}$ are i.i.d. random variables with $\mb{P}(\tau_{i}= \pm 1)  = 1/2$, independent of $\Xi$. Taking the empirical measure $P_{m} = m^{-1} \sum_{i=1}^{m} \delta_{(\xi_i,r_i)}$ and noting that
\begin{align*}
\sup _{g \in \bar{G}}\|g\|_{L^{2}(P_{m})}^2 & = \sup _{x \in \Omega}\frac{1}{m}\sum_{i=1}^m (1+\abs{\xi_i}_1)^{2s}[\gamma_{,n_{\xi_{i}}}(t_{\xi_{i}}(x),r_{i})]^{2}\leq \frac{1}{16m}\sum_{i=1}^m (1+\abs{\xi_i}_1)^{2s}.
\end{align*}
By Lemma~\ref{Lemma: Metric entropy Bound for empirical Rademacher complexity}, we obtain for any $\delta > 0$, 
\begin{equation} \label{ineq: empirical complexity bound}
\widehat{\mathcal{R}}_{m}(\bar{G},\Xi) \leq  \delta + \frac{\sqrt{2}}{4\sqrt{m}} \left(\frac{1}{m}\sum_{i=1}^m (1+\abs{\xi_i}_1)^{2s} \right)^{1/2} \sqrt{\ln[2\mathcal{N}\!(\delta, G, \|\cdot\|_{L^{1}(P_{m})})]} ,
\end{equation}
where we have used the relation $\mathcal{N}(\delta, \bar{G}, \|\cdot\|_{L^{1}(P_{m})}) \leq 2 \mathcal{N}(\delta, G, \|\cdot\|_{L^{1}(P_{m})})$ for $\bar{G} = G \cup (-G)$.

Since $\gamma_{,n_\xi}(t_\xi(\cdot),r)$ is a $\abs{\xi}_1$-Lipschitz continuous function, i.e., 
\[
\left|\gamma_{,n_\xi}(t_\xi(x),r) - \gamma_{,n_\xi}(t_\xi(y),r)\right| \leq \abs{\xi}_1 \abs{x-y}_{\infty}.
\]
Hence, for any $x,y\in \Omega$, 
\begin{equation*}
\begin{aligned}
\|g_{x} - g_{y}\|_{L^{1}(P_{m})} & = \frac{1}{m}\sum_{i=1}^m (1+\abs{\xi_i}_1)^{s}\left|\gamma_{,n_{\xi_{i}}}(t_{\xi_{i}}(x),r_{i}) - \gamma_{,n_{\xi_{i}}}(t_{\xi_{i}}(y),r_{i})\right| \\
& \leq \frac{1}{m}\sum_{i=1}^m (1+\abs{\xi_i}_1)^{1+s} \abs{x-y}_{\infty}.
\end{aligned}
\end{equation*}
This gives
\begin{equation*}
\begin{aligned}
\mathcal{N}\left(\delta,G, \|\cdot\|_{L^{1}(P_{m})}\right)&\leq\mathcal{N}\left(\delta m \left(\sum_{i=1}^m (1+\abs{\xi_i}_1)^{1+s}\right)^{-1}, \Omega, \abs{\cdot}_{\infty}\right)\\
&\leq\left\lceil \frac{1}{2\delta m} \sum_{i=1}^m (1+\abs{\xi_i}_1)^{1+s}\right\rceil^{d}.
\end{aligned}
\end{equation*}
Taking $\delta = (60\sqrt{m})^{-1}$ and substituting the above estimate into \eqref{ineq: empirical complexity bound}, we obtain
\[
\widehat{\mathcal{R}}_{m}(\bar{G},\Xi) 
\le \delta + \frac{\sqrt{2}}{4\sqrt{m}} \sqrt{\frac{1}{m}\sum_{i=1}^m (1\!+\!\abs{\xi_i}_1)^{2s} }  \sqrt{\ln 2 + d \ln\! \left( \frac{31}{\sqrt{m}} \sum_{i=1}^m (1\!+\!\abs{\xi_i}_1)^{1+s}\right)}. 
\]
Note that $2s/(1+s)<1$, 
\begin{equation*}
\begin{aligned}
\mb{E}_{\Xi}\ln \left( \sum_{i=1}^m (1+\abs{\xi_i}_1)^{1+s}\right) & \le  \frac{1+s}{2s} \mb{E}_{\Xi}\ln \left( \sum_{i=1}^m (1+\abs{\xi_i}_1)^{2s} \right)\\
&\le \frac{1+s}{2s} \ln \left(m \frac{\|f\|_{\ms{B}^s(\Omega)}}{Q} \right) ,
\end{aligned}
\end{equation*}
where the second line follows from Jensen's inequality and \[
\mb{E}_{\Xi}\sum_{i=1}^m(1+\abs{\xi_i}_1)^{2s} = m\|f\|_{\ms{B}^s(\Omega)} /Q.
\]
Therefore, by Cauchy's inequality, 
we get 
\begin{equation*}
\begin{aligned}
\mb{E}_{\Xi} \widehat{\mathcal{R}}_{m}(\bar{G},\Xi) 
& \le \frac{1}{60\sqrt{m}} + \frac{\sqrt{2 \mb{E}_{\xi} (1\!+\!\abs{\xi}_1)^{2s}}}{4\sqrt{m}}  \sqrt{\ln 2 + d \mb{E}_{\Xi}\ln  \left( \frac{31}{\sqrt{m}} \sum_{i=1}^m (1\!+\!\abs{\xi_i}_1)^{1+s}\right)} \\
& \le \frac{1}{60\sqrt{m}} + \frac{\sqrt{2d \|f\|_{\ms{B}^s(\Omega)}}}{4\sqrt{Q m}}  \sqrt{\ln 62 + \frac{\ln m}{2s} + \frac{1+s}{2s} \ln \left(\frac{\|f\|_{\ms{B}^s(\Omega)}}{Q} \right) } .
\end{aligned}
\end{equation*}
By definition, 
$\|f\|_{\ms{B}^s(\Omega)} /Q \ge 1.$
Hence, by $1/15 + \sqrt{ 2\ln 62} C \leq  \sqrt{2.094 \ln 62} C$ for $C\ge 1$, we absorb the first term into the second one and get
\[
\mb{E}_{\Xi} \widehat{\mathcal{R}}_{m}(\bar{G},\Xi) 
\leq \frac{\sqrt{2.1 d \|f\|_{\ms{B}^s(\Omega)}}}{4\sqrt{Q m}}  \sqrt{\ln 62 + \frac{\ln m}{2s} + \frac{1+s}{2s} \ln \left(\frac{\|f\|_{\ms{B}^s(\Omega)}}{Q} \right) }  . 
\]
Substituting the above estimate into \eqref{ineq: symmetrization}, we obtain 
\begin{equation*}
\begin{aligned}
\mb{E}_{\Xi}\|f-f_{\Xi}\|_{L^{\infty}(\Omega)} 
& \leq  \frac{\pi \sqrt{2.1 d \|f\|_{\ms{B}^s(\Omega)} Q }}{\sqrt{2s m}} \sqrt{\ln 62m + (1+s) \ln \left(\frac{\|f\|_{\ms{B}^s(\Omega)}}{Q} \right) }   .
\end{aligned}
\end{equation*}
By Markov's inequality, with probability at least $(1+\varepsilon)/(2+\varepsilon)$, for some $\varepsilon>0$ to be chosen later on, we have
\begin{equation}\label{eq: Linf error}\begin{aligned}
\|f-f_{\Xi}\|_{L^{\infty}(\Omega)} & \le   \frac{\pi (2+\varepsilon)\sqrt{2.1d \|f\|_{\ms{B}^s(\Omega)} Q }}{\sqrt{2s m}} \sqrt{\ln 62m + \frac{3}{2} \ln \left(\frac{\|f\|_{\ms{B}^s(\Omega)}}{Q} \right) } \\
& \le  \frac{\pi (2+\varepsilon)\sqrt{6.3d}\|f\|_{\ms{B}^s(\Omega)}}{2\sqrt{s\tilde{c}}N^{sL}} \sqrt{\ln \left(62^{2/3}(\tilde{c}+2^{-10})N^{2sL} \right) },\\
&\le \frac{\pi (2+\varepsilon)\sqrt{6.3dL}\|f\|_{\ms{B}^s(\Omega)}}{\sqrt{2\tilde{c}}N^{sL}} \sqrt{\ln \left(c_0N \right) },
\end{aligned}\end{equation}
where $c_0^{2sL}:=62^{2/3}(\tilde{c}+2^{-10})$ and the second inequality follows from~\eqref{ineq: upper bound for m}. Setting $\tilde{c} = 16.26^{-1}$ and $\varepsilon=10^{-4}$, we obtain~\eqref{eq: Linf approximation error by deep NN} since $c_0<1$.

It remains to calculate the number of units in each layer that whether $N_\Xi\le N$. 
Since the construction of $f_{\Xi}$ is the same as in the proof of Theorem~\ref{Theorem: Lp approximation by deep NN}, by~\eqref{eq: width bound} and~\eqref{ineq: upper bound for m}, we obtain, with probability at least $(1+\varepsilon)/(2+\varepsilon)$,
\begin{equation}\label{ineq: width bound, Linf DNN}
N_{\Xi}^{2sL} \le 8 (2+\varepsilon)\frac{m\|f\|_{\ms{B}^s(\Omega)}}{Q} \le 8 (2+\varepsilon) (\tilde{c} + 2^{-10}) N^{2sL}\le N^{2sL}.
\end{equation}
A combination of~\eqref{eq: Linf error} and~\eqref{ineq: width bound, Linf DNN} yields the desired estimate with probability at least $\varepsilon/(2+\varepsilon)$, which completes the proof. 
\end{proof}

As before, we derive an error bound involving only the seminorm $\upsilon_{f,s}$. Notably, in this proof, the intermediate prefactor $c_0$ of the logarithmic factor tends to zero as $sL \to 0$. This reflects a modest improvement in the logarithmic factor, consistent with insights from the theory of real interpolation.

%
\begin{corollary} \label{Corollary: Linf approximation bounds with spectral semi-norm DNN}
Let $L\in \mathbb{N}_{+}$ and $f\in\ms{B}^s(\Omega)$ with $0<sL\le 1/2$. For any positive integer $N$, there exists an $(L,N)$-network $f_{L,N}$ such that
\begin{equation}\label{eq:coroLinf}
\|f-f_{L,N}\|_{L^{\infty}(\Omega)} \le   C\upsilon_{f,s}\Bigl(\frac{1+dL\ln N}{N}\Bigr)^{sL},
\end{equation}
where $C$ is a universal constant. 
\end{corollary}

\begin{proof}
Recalling the frequency decomposition in~\eqref{eq: decompose f into low and high frequency parts}, and we approximate the low-frequency component using Lemma~\ref{lemma:Linf}, and the high-frequency component using Theorem~\ref{Theorem: Linf approximation by deep NN}, respectively. Proceeding along a similar approach that leads to Corollary~\ref{Corollary: approximation bounds with spectral semi-norm DNN}, we obtain, for $f\in\ms{B}^{1/(2L)}(\Omega)$, there exists an $(L,N)$-network $f_{L,N}$ such that
\begin{equation}\label{eq:coroLinf1}
\|f-f_{L,N}\|_{L^{\infty}(\Omega)} \le C\upsilon_{f,1/(2L)} \Lr{\dfrac{1+dL\ln N}{N}}^{1/2}.
\end{equation}
The details are omitted for brevity.

Next for general $f\in\ms{B}^s(\Omega)$ with $0<sL\le 1/2$, we decompose $f=f_1+f_2$ with
\[
f_1(x)=\operatorname{Re}\int_{\abs{\xi}_1<R}\wh{f}(\xi)e^{2\pi i\xi\cdot x}\mr{d}\xi,\quad f_2(x)=\operatorname{Re}\int_{\abs{\xi}_1\ge R}\wh{f}(\xi)e^{2\pi i\xi\cdot x}\mr{d}\xi,
\]
where $R:=N^L(1+dL\ln N)^{-L}$. It is straightforward to verify
\[\wh{f}_1(\xi) =\frac{1}{2}[\wh{f}(\xi)+\wh{f}(-\xi)^{*}]\chi_{[0,R)}(\abs{\xi}_1), \qquad\wh{f}_2(\xi)=\frac{1}{2}[\wh{f}(\xi)+\wh{f}(-\xi)^{*}]\chi_{[R,\infty)}(\abs{\xi}_1),
\]
where $z^{*}$ denotes the complex conjugate of $z$. Hence $f_1\in\ms{B}^{1/(2L)}(\Omega)$ because
\begin{align*}
\upsilon_{f_1,1/(2L)}&=\int_{\abs{\xi}_1<R}\abs{\xi}_1^{1/(2L)}\abs{\wh{f}(\xi)}\mr{d}\xi\le R^{1/(2L)-s}\int_{\mb{R}^d}\abs{\xi}_1^{s}\abs{\wh{f}(\xi)}\mr{d}\xi\\
&=N^{1/2-sL}(1+dL\ln N)^{sL-1/2}\upsilon_{f,s}.
\end{align*}
By~\eqref{eq:coroLinf1}, there exists an $(L,N)$-network $f_{L,N}$ such that
\[
\|f_1-f_{L,N}\|_{L^{\infty}(\Omega)} \le   C\upsilon_{f_1,1/(2L)} \Lr{\frac{1+dL\ln N}{N}}^{1/2}\le C\upsilon_{f,s}\Bigl(\frac{1+dL\ln N}{N}\Bigr)^{sL}.
\]
Meanwhile, using the Hausdorff-Young inequality,
\[
\nm{f_2}{L^\infty(\Omega)}\le\upsilon_{f_2,0}=\int_{\abs{\xi}_1\ge R}\abs{\wh{f}(\xi)}\mr{d}\xi\le\frac{1}{R^{s}}\int_{\mb{R}^d}\abs{\xi}_1^s\abs{\wh{f}(\xi)}\mr{d}\xi=\upsilon_{f,s}\Bigl(\frac{1+dL\ln N}{N}\Bigr)^{sL}.
\]
By the triangle inequality and the above two inequalities, we obtain~\eqref{eq:coroLinf}.
\end{proof}

Proceeding along the same line that leads to~\eqref{eq:coroLinf}, we may improve Corollary~\ref{Corollary: approximation bounds with spectral semi-norm DNN} to
\[
\nm{f-f_{L,N}}{L^p(\Omega)} \le   C\upsilon_{f,s}\Bigl(\dfrac{p}{N}\Bigr)^{sL}.
\]
We omit the details for brevity.

%% file: Section/Lower_bound.tex
\subsection{Lower bounds} 
In what follows, we show that the $L^{p}$-approximation rates we obtained in Section~\ref{Section: L^p approximation by shallow NN} and Section~\ref{Section: deep neural network approximation} are sharp, and the uniform approximation rate in Section~\ref{Section: deep neural network approximation} is sharp up to a logarithmic factor. This example is adopted from~\cite[Theorem 2]{BreslerNagaraj:2020} for the sharpness of $L^2$-approximation. We adopt it here for $L^p$-approximation with $1\le p\le\infty$.  
\begin{theorem}\label{Theorem: lower bound for DNN approximation}
Let $1\le p\le\infty$. For any fixed positive integers $L,N$ and real numbers $\varepsilon,s$ with $0<\varepsilon,sL\le 1/2$, there exists $f\in\ms{B}^{s}(\Omega)$ satisfying $\upsilon_{f,s}\le1+\varepsilon$ and $\|f\|_{\ms{B}^{s}(\Omega)} \le 2+\varepsilon$ such that for any deep ReLU or Heaviside neural network $f_{L,N}$ with $L$ hidden layer and $N$ units each layer,
\begin{equation}\label{eq:lower}
\|f-f_{L,N}\|_{L^p(\Omega)}\ge\dfrac{1-\varepsilon}{4\sqrt{2}\pi N^{sL}}.
\end{equation}
\end{theorem}

There are approximation lower bounds of neural networks to Barron functions provided in \cites{Makovoz1996Random,Barron:2018,SiegelXu:2022} for SNN and \cite{Achour2022generalLplbd} for DNN. Their proofs share the key idea of comparing the metric entropy of the target space and the network class, leading to the rates $1/2+\mathcal{O}(s/d)$ for SNN and $1+\mathcal{O}(s/d)$ for DNN (up to logarithmic factors) with respect to the width of the network. For small $s$, this approach suggests the limit of $1/2$ for the approximation rate, but they are not sharp in this range. We prove a sharp rate $sL$ as $0<sL\leq1/2$ by accurately identifying the worst case scenario and leveraging the idea of oscillation counting. 

To prove Theorem~\ref{Theorem: lower bound for DNN approximation}, we cite a useful lemma that estimates the spectral norm of a series of Schwartz functions.
\begin{lemma}[{\cite[Lemma SM6.1]{LiaoMing:2023}}]  \label{lema:decay}
Given $n,R>0$ and consider $$f(x)=\cos(2\pi nx_1)e^{-\pi\abs{x}^2/R},$$ then $f\in\ms{B}^{s}(\Omega)$ with $\upsilon_{f,0} = 1$ and
	\begin{equation}\label{eq:decay}
		\upsilon_{f,s} \le \Lr{n+\dfrac{d}{\pi\sqrt{R}}}^s\qquad\text{for}\quad 0\le s\le 1.
	\end{equation}
\end{lemma}

\begin{proof}[Proof of Theorem~\ref{Theorem: lower bound for DNN approximation}]
Define $n=2^{L+2}N^L$ and $f(x)=n^{-s}\cos(2\pi nx_1)e^{-\pi\abs{x}^2/R}$ with large enough $R$ such that $\upsilon_{f,s}\le 1+\varepsilon$ by Lemma~\ref{lema:decay} and $e^{-\pi\abs{x}^2/R}\ge 1-\varepsilon$ when $x\in\Omega$. Then, 
\[
\|f\|_{\ms{B}^{s}(\Omega)} \le \upsilon_{f,0}+ \upsilon_{f,s}\le 2+\varepsilon.
\]
Fix $\hat{x}{:}=(x_2,\dots,x_d)$, then $f_{L,N}$ can be viewed as an one-dimensional neural network, i.e. $f_{L,N}(\cdot,\hat{x}_2):[0,1]\to\mb{C}$. Divide $[0,1]$ into $n$-internals of $[j/n,(j+1)/n]$ with $j=0,\dots,n-1$. There exists at least $n-2^{L+1}N^L=2^{L+1}N^L$ intervals such that $f_{L,N}$ does not change sign on those intervals~\cite[Lemma 3.2]{Telgarsky:2016}. Without loss of generality, we assume $f_{L,N}(\cdot,\hat{x}_2)\ge 0$ on some interval $[j/n,(j+1)/n]$, then
\[
\int_{j/n}^{(j+1)/n}|f(x)-f_{L,N}(x)|\mr{d}x_1\ge\dfrac{1-\varepsilon}{n^{s}} \int_{(4j+1)/(4n)}^{(4j+3)/(4n)}|\cos(2\pi nx_1)|\mr{d}x_1\ge\dfrac{1-\varepsilon}{\pi n^{s+1}},
\]
because $\cos(2\pi nx_1)\le 0$ when $2\pi j+\pi/2\le 2\pi nx_1\le 2\pi j+3\pi/2$.

Summing up these $n-2^{L+1}N^L$ intervals gives
\begin{align*}
\|f-f_{L,N}\|_{L^1(\Omega)}&\ge
\int_{[0,1]^{d-1}}\mr{d}\hat{x}\int_0^1|f(x)-f_{L,N}(x)|\mr{d}x_1\\
&\ge \dfrac{2^{L+1}N^L(1-\varepsilon)}{\pi n^{s+1}}\\
&\ge \dfrac{1-\varepsilon}{2^{s(L+2)+1} \pi N^{sL}}\\
&\ge\dfrac{1-\varepsilon}{4\sqrt{2} \pi N^{sL}}.
\end{align*}
The proof is completed by noting that for all $p\ge 1$,
\[
\|f-f_{L,N}\|_{L^p(\Omega)}\ge\|f-f_{L,N}\|_{L^1(\Omega)}  .
\]
\end{proof}

%% file: Appendix/bound_for_Gamma_function.tex
\section{Proof of Lemma~\ref{Lemma: Bound for Gamma function}} 
\label{Appendix Section: Bound for Gamma function}


\begin{proof}
The proof is divided into two parts, dealing with the cases $p \in [2, 7/3)$ and $p \in [7/3, +\infty)$, respectively.

For $p \in [2, 7/3)$, let $x = (p+1)/2$. We claim that for $x \in [3/2, 5/3)$, 
$$h(x) :=\frac{1}{2 x-1} \ln \Gamma(x)-\frac{\ln \pi}{2(2 x-1)}-\frac{1}{2} \ln (2 x-1) $$
is monotonically decreasing. 

Denote by $\psi(x)=(\ln \Gamma(x))^{\prime}= \Gamma^{\prime}(x)/\Gamma(x)$ the digamma function. The derivative of $h$ is 
\[
h^{\prime}(x) =\frac{(2 x-1)\left(\psi(x)-1\right)-2 \ln \Gamma(x)+\ln \pi}{(2x-1)^2} .
\]
Since
\[\Gamma^{(k)}(x)=\int_{0}^{\infty} t^{x-1} e^{-t}(\ln t)^{k}\mr{d}t, \qquad k\in\mathbb{N},
\]
by Cauchy-Schwarts inequality, $(\Gamma^{\prime}(x))^2 \le \Gamma^{\prime\prime}(x)\Gamma(x),$ which implies $\psi^{\prime}(x) \ge 0$ and that $\psi$  is monotonically nondecreasing. 
Using the integral representation of the digamma function~\cite[eq. (2.10)]{Artin1964gamma}, 
$$\psi(s+1)=\int_{0}^{1} \frac{1-t^{s}}{1-t} \mr{d} t-\gamma, \qquad \text{for any } s\ge 0, $$
where $\gamma\approx0.58$ is the Euler constant. A direct calculation gives the values of $\psi$ at the endpoints
\[
\psi(3/2) =\int_{0}^{1} \frac{1-\sqrt{t}}{1-t} \mr{d} t-\gamma = 2 - 2\ln 2 -\gamma,
\]
and
\begin{align*}
\psi(5/3) & =\int_{0}^{1} \frac{1-t^{2/3}}{1-t} \mr{d} t-\gamma \\
& =\left.\left[\frac{3}{2}\left(t^{2 / 3}-\ln(t^{2 / 3}+\sqrt[3]{t}+1)\right) +\sqrt{3} \tan ^{-1}\left(\frac{2 \sqrt[3]{t}+1}{\sqrt{3}}\right)\right]\right|_{0} ^{1}-\gamma \\
& =\frac{1}{6}(9+\sqrt{3} \pi-9 \ln 3)-\gamma, 
\end{align*}
where $\tan^{-1}$ is the inverse function of the tangent function. Since $\psi(3/2) > 0.6 - \gamma>0$ and $\psi$ is monotonically nondecreasing,  $\psi(x) > 0$ and $\Gamma(x)$ monotonically increase for all $x \ge 3/2$, which implies 
$$-2\ln \Gamma(x)+\ln \pi \le -2\ln \Gamma(3/2)+\ln \pi =2\ln2, \quad \text{for } x\in [3/2,5/3].$$ 
Since $\psi(5/3) < 0.76 - \gamma < 1$ and $\psi$ is monotonically nondecreasing, for $x\in [3/2,5/3]$, $\psi(x) - 1 < 0$ and  $$(2 x-1)(\psi(x)-1) \le 2(\psi(x)-1) 
<-2(\gamma+0.24).$$
Then, we obtain $h^{\prime}(x) < 0$ from $$(2 x-1)\left(\psi(x)-1\right)-2 \ln \Gamma(x)+\ln \pi \le 2(\ln 2-\gamma-0.24)<0.$$ 
The claim is proved, which immediately implies that for $p\in [2, 7/3)$,
\[
\pi^{-1/(2p)}  \Gamma\left((p+1)/2\right)^{1/p} p^{-1/2}=\exp(h((p+1)/2))
\leq \exp(h(3/2)) = 1/2.
\]

Next, for $p\ge 7/3,$ we leverage Stirling's approximation~\cite[eq. (3.9)]{Artin1964gamma},
$$\Gamma(x) \le \sqrt{2 \pi} x^{x-1/2} e^{-x+1 /(12 x)}\quad \text{for all } x>0.$$
With the choice $x = (p+1)/2$, we obtain
$$\pi^{-1/(2 p)} \Gamma\left(\frac{p+1}{2}\right)^{1 / p} p^{-1/2} \le \left(\frac{2}{e}\right)^{1 /(2 p)}\left(\frac{p+1}{2 e p}\right)^{1/2} e^{1/[6 p(p+1)]}.$$
We denote the logarithm of the right side of the above inequality as 
$$g(p) = \frac{1}{2 p} \ln \left(\frac{2}{e}\right)+\frac{1}{2} \ln \left(\frac{p+1}{2 e p}\right)+\frac{1}{6 p(p+1)}. $$
Then, $g(p)$ monotonically decrease because 
\begin{equation*}
\begin{aligned}
g^{\prime}(p) 
& =\frac{1}{2 p^{2}}\left[\ln \left(\frac{e}{2}\right)-\frac{1}{3}-\dfrac{p(2p+3)}{3(p+1)^{2}}\right]<0.
\end{aligned}
\end{equation*}
We may conclude that, for $p\ge 7/3$,
\[
e^{g(p)} \le e^{g(7/3)}=2^{3/14} \sqrt{5/7} e^{-97/140}<1/2,
\]
which completes the proof. 
\end{proof}
\section{The sharpness of the embedding relation \texorpdfstring{$\ms{B}^s \hookrightarrow C^s$}{}}~\label{AppendixB}

By~\cite[Corollary 2.13]{LiaoMing:2023}, $\ms{B}^s(\mb{R}^{d}) \hookrightarrow C^s(\mb{R}^{d})$, where $C^s(\mb{R}^{d})$ denotes the H\"older space of order $s$. 
We shall construct an example to illustrate that this embedding is essentially sharp, and generally, the H\"older continuity index of the $\ms{B}^s$-functions can not be higher than $s$.
Therefore, the target function space in this paper, $\ms{B}^s$ with small $s$, may contain many rough functions.  
\begin{lemma} \label{Lemma: Barron and Holder regularity of Bessel potentials}
Let $\alpha \in (0,1)$ and $f:\mb{R}^{d}\to \mb{R}$ be the Bessel potential of order $\alpha +d$
\begin{equation} \label{eq: def Bessel potential}
f(x)=\int_{\mathbb{R}^{d}}\left(1+4 \pi^{2}|\xi|^{2}\right)^{-(\alpha+d) / 2} e^{i 2 \pi \xi \cdot x} \mr{d} \xi .
\end{equation}
Then, 
\begin{enumerate}
\item[(1)]
$f \in \ms{B}^{s}(\mb{R}^{d})$ if and only if $s<\alpha$; 

\item[(2)] $f \in C^{s}(\mb{R}^{d})$ if and only if $s \le \alpha$. 
\end{enumerate}
\end{lemma}

\begin{proof}
(1) For any vector $\xi$, we denote the bracket $\langle\xi\rangle = (1+|\xi|^{2})^{1/2}$. By~\cite{Stein:1971}, 
\[
\wh{f}(\xi) = (1+4 \pi^{2}|\xi|^{2})^{-(\alpha+d) / 2}.
\]
It follows from $(1+|\xi|_{1})^{2} \le 2 d\langle\xi\rangle^{2}$ that
\[
\|f\|_{\ms{B}^{s}(\mb{R}^{d})}  = \int_{\mb{R}^{d}}(1+|\xi|_{1})^{s}\langle2\pi\xi\rangle^{-\alpha-d} \mr{d} \xi\le (2 d)^{s / 2} \int_{\mb{R}^{d}}\langle\xi\rangle^{s-\alpha-d} \mr{d} \xi ,
\]
which implies $f \in \ms{B}^{s}(\mb{R}^{d})$ when $s<\alpha$. Similarly, by $1+|\xi|_{1} \ge \langle2\pi\xi\rangle/(2 \pi)$, we have the lower bound 
\begin{equation*}
\begin{aligned}
\|f\|_{\ms{B}^{s}(\mb{R}^{d})} & \ge (2 \pi)^{-s} \int_{\mb{R}^{d}}\langle2\pi\xi\rangle^{s-\alpha-d} \mr{d} \xi ,
\end{aligned}
\end{equation*}
which implies $f \notin \ms{B}^{s}(\mb{R}^{d})$ when $s \ge \alpha$. 

(2) Note that $f$ is a radial function. First, we derive a representation of $f$ in terms of $\abs{x}$.
For $d \ge 2$, denote $\xi = (\xi_{1},\xi^{\prime})$ with $\xi^{\prime} \in \mb{R}^{d-1}$. By rotation invariance, a direct manipulation yields
\begin{equation*}
\begin{aligned}
f(x) 
& = \frac{1}{(2\pi)^d}\int_{\mb{R}^{d}}\left(\langle\xi_1\rangle^{2} + |\xi^{\prime}|^{2}\right)^{-(\alpha+d) / 2} e^{i |x|  \xi_{1}} \mr{d} \xi \\
& = \frac{\omega_{d-1}}{(2\pi)^d}\int_{\mb{R}}\left[\int_{0}^{\infty}\left(\langle\xi_1\rangle^{2} + r^{2}\right)^{-(\alpha+d) / 2}r^{d-2} \mr{d} r\right] e^{i |x| \xi_{1}} \mr{d} \xi_{1} ,
\end{aligned}
\end{equation*}
where $\omega_{d-1} = 2\pi^{(d-1)/2}/ \Gamma((d-1)/2)$ denotes the surface area of $(d-1)$-dimensional unit ball. With variable substitution $t= r^{2}/(\langle\xi_1\rangle^{2} + r^{2})$, the inner integral with respect to $r$ can be expressed as Euler’s beta integral
\begin{equation*}
\begin{aligned}
\frac{1}{2\langle\xi_1\rangle^{\alpha+1}} \int_{0}^{1} t^{(d-3) / 2}(1-t)^{(\alpha-1) / 2} d t
= \frac{\Gamma((d-1)/2) \Gamma((\alpha+1)/2)}{2\Gamma((\alpha+d)/2) \langle\xi_1\rangle^{\alpha+1}} , 
\end{aligned}
\end{equation*}
which implies
\begin{equation} \label{eq: radial function representation of f}
\begin{aligned}
f(x) & =  \frac{\Gamma((\alpha+1)/2)}{2^{d}\pi^{(d+1)/2} \Gamma((\alpha+d)/2)} \int_{\mb{R}}\langle\xi\rangle^{-\alpha-1} e^{i |x| \xi} \mr{d} \xi .
\end{aligned}
\end{equation}
The same expression also holds for $d=1$.

Next, we prove $f \in C^{s}(\mb{R}^{d})$ for $s \le \alpha$. 
For any $x,y\in \mb{R}^{d}$, let $z = \left||x|-|y|\right|$ and then $z\le |x-y|$. If $z \ge 1$, it follows from Hausdorff-Young inequality that  $$|f(x)-f(y)| \le \|f\|_{L^{\infty}(\mb{R}^{d})}\le \|f\|_{\ms{B}^{0}(\mb{R}^{d})}|x-y|^{s}.$$
If $z=0$, then \eqref{eq: radial function representation of f} yields $|f(x)-f(y)| = 0$.  

If $0 < z < 1$, then by \eqref{eq: radial function representation of f}, we have
\begin{equation*} \label{ineq: first upper bound for |f(x)-f(y)|}
\begin{aligned}
|f(x)-f(y)| & \le \frac{\Gamma((\alpha+1)/2)}{2^{d}\pi^{(d+1)/2} \Gamma((\alpha+d)/2)} \int_{\mathbb{R}} \langle\xi_{1}\rangle^{-\alpha-1}\left|e^{i z \xi}-1\right| \mr{d} \xi  \\
& \le \frac{\Gamma((\alpha+1)/2)}{2^{d-2}\pi^{(d+1)/2} \Gamma((\alpha+d)/2)} \int_{0}^{\infty} \langle\xi\rangle^{-\alpha-1}\left|\sin(z \xi/2)\right| \mr{d} \xi ,
\end{aligned}
\end{equation*}
where we have used $\abs{e^{i z \xi}-1}^{2}=4 \sin^{2}(z \xi/2)$ in the last equality. 

Using $\left|\sin t\right|\le \max(|t|,1)$ and $\langle\xi\rangle^{-1}\le \sqrt{2}(1+|\xi|)^{-1}$, we bound the above integral by
\begin{equation*}
\begin{aligned}
 \int_{0}^{\infty} \langle\xi\rangle^{-\alpha-1}\left|\sin(z \xi/2)\right| \mr{d} \xi 
& \le \frac{z}{2} \int_{0}^{\pi/z} \langle\xi\rangle^{-\alpha-1} \xi \mr{d} \xi +  \int_{\pi/z}^{\infty} \langle\xi\rangle^{-\alpha-1} \mr{d} \xi \\
& \le \frac{2^{(\alpha-1) / 2}z}{ 1-\alpha}  \left(1+\pi/z\right)^{1-\alpha} +  \frac{2^{(\alpha+1) / 2}}{\alpha} \left(1+\pi/z\right)^{-\alpha}, \\
\end{aligned}
\end{equation*}
which together with the fact $1+\pi/z \le 2\pi/z$ gives
\begin{equation*}
\begin{aligned}
|f(x)-f(y)| & \le \frac{\Gamma((\alpha+1)/2)}{2^{d-5/2}\pi^{(d+1)/2+\alpha} \Gamma((\alpha+d)/2)} \left(\frac{2^{-\alpha / 2} \pi}{1-\alpha}  +  \frac{2^{\alpha / 2}}{\alpha} \right) z^{\alpha}.  
\end{aligned}
\end{equation*}
Note that $z^{\alpha}\le z^{s}$ when $z<1$. A combination of the estimates for $z\ge 1$, $z=0$ and the above estimate yields that $f \in C^{s}(\mb{R}^{d})$ for $s \le \alpha$. 

Next, it is sufficient to prove that $|f(x)-f(0)|$ asymptotically behaves as $|x|^{\alpha}$ near the origin, which implies $f \notin C^{s}(\mb{R}^{d})$ for $s > \alpha$. Let $x\in\mb{R}^{d}$ and $0<|x|<1$. It follows from \eqref{eq: radial function representation of f}, evenness and $1-\cos(|x| \xi) \ge 0$ that 
\begin{equation*} 
\begin{aligned}
|f(x)-f(0)| & =  \frac{\Gamma((\alpha+1)/2)}{2^{d-1}\pi^{(d+1)/2} \Gamma((\alpha+d)/2)} \int_{0}^{\infty}\langle\xi\rangle^{-\alpha-1} \left(1-\cos(|x| \xi)\right) \mr{d} \xi \\
& \ge \frac{\Gamma((\alpha+1)/2)}{2^{d-1}\pi^{(d+1)/2} \Gamma((\alpha+d)/2)} \int_{\pi/(2|x|)}^{\infty}(1+\xi)^{-\alpha-1} \left(1-\cos(|x| \xi)\right) \mr{d} \xi ,
\end{aligned}
\end{equation*}
where the first part in the integral has an exact lower bound
\begin{equation*} 
\begin{aligned}
\int_{\pi/(2|x|)}^{\infty}(1+\xi)^{-\alpha-1} \mr{d} \xi = \frac{1}{\alpha}\left(1+\frac{\pi}{2|x|}\right)^{-\alpha} \ge \frac{|x|^{\alpha}}{\alpha\pi^{\alpha}}.
\end{aligned}
\end{equation*}
For the other part, we split the integral into the sum over each period and obtain 
\begin{equation*} 
\begin{aligned}
& \quad \int_{\pi/(2|x|)}^{\infty}(1+\xi)^{-\alpha-1} \cos(|x| \xi) \mr{d} \xi\\
&= \sum_{k=0}^{\infty} \int_{(1+4k)\pi/(2|x|)}^{(3+4k)\pi/(2|x|)} \left[(1+\xi)^{-\alpha-1} - (1+\xi+\pi/|x|)^{-\alpha-1} \right] \cos(|x| \xi) \mr{d} \xi ,
\end{aligned}
\end{equation*}
which is non-positive because $t^{-\alpha-d}$ monotonically decreases and $\cos(|x| \xi) \le 0$ for $|x| \xi\in [(1+4k)\pi/2,(3+4k)\pi/2]$. The proof completes by combining the above three estimates. 
\end{proof}

\begin{figure}[htbp]
  \centering
  \includegraphics[width=0.49\textwidth]{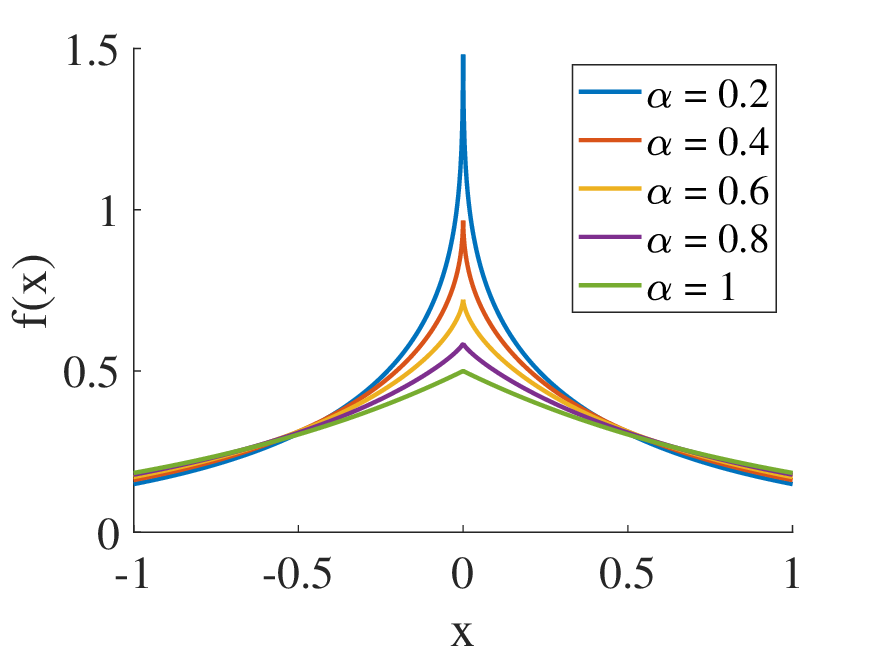}
  \includegraphics[width=0.5\textwidth]{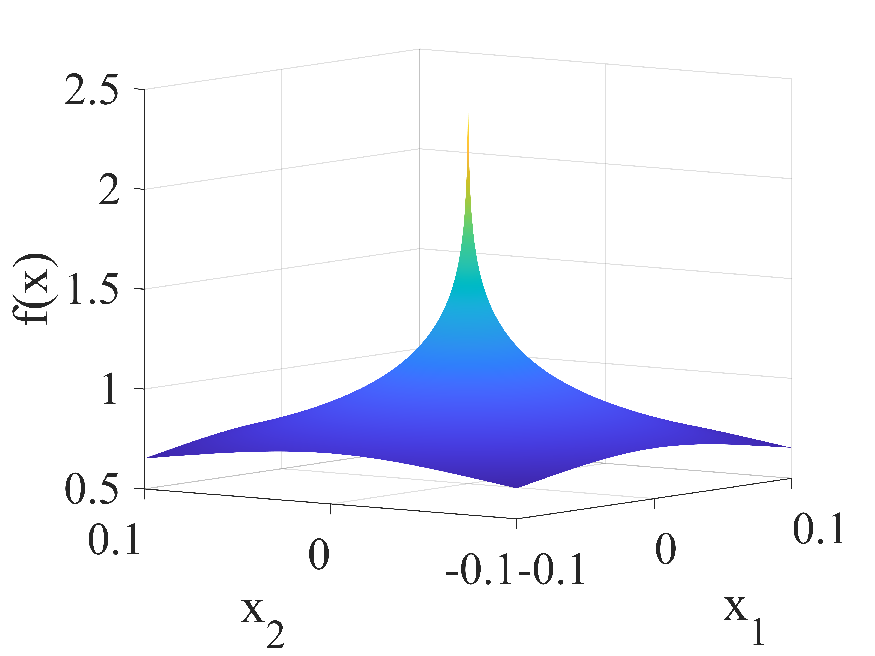}
\caption{The plots of $f$ defined in~\eqref{eq: def Bessel potential}. Left: $d=1$ with different $\alpha$. Right: $d=2$ and $\alpha=0.05$.} 
  \label{Figure: Bessel potential}
\end{figure}

By~\eqref{eq: radial function representation of f},  $f$ defined in \eqref{eq: def Bessel potential} is a radial function with the representation
\[
f(x)
=  \left[2^{d+\alpha/2-1}\pi^{d/2} \Gamma((\alpha+d)/2)\right]^{-1}  |x|^{\alpha/2}  K_{\alpha/2}(|x|),
\]
where $K_{\nu}$ is modified Bessel function of the second kind
\[
K_{\nu}(z) = \frac{\Gamma(\nu+1 / 2)(2 z)^{\nu}}{\sqrt{\pi}} \int_{0}^{\infty} \left(t^{2}+z^{2}\right)^{-\nu-1 / 2} \cos t \ \mr{d} t.
\]
If $\alpha=1$, then $f$ has an explicit expression as
\[
f(x) = 2^{-d}\pi^{(1-d)/2} \Gamma((d+1)/2)^{-1} e^{-|x|},
\]
because $K_{1/2}(z)=\sqrt{\pi/(2z)}e^{-z}$~\cite[\S 1.4.6, eq. (9)]{Luke:1962}. 
To visualize the singularity of $f$ near the origin, we plot $f$ for $d=1$ with different values of $\alpha$, and for $d=2$ with $\alpha=0.05$ in Figure~\ref{Figure: Bessel potential}. 
As proved in Lemma~\ref{Lemma: Barron and Holder regularity of Bessel potentials}, $f(0)-f(x)$ asymptotically behaves as $\abs{x}^{\alpha}$. The smaller the $\alpha$, the stronger the singularity. For $\alpha=0.05$, $f \in \ms{B}^{s}(\mb{R}^{d})$ with $s<0.05$. The results in Section~\ref{Section: deep neural network approximation} ensure that the networks of $2$, $6$ and $10$ hidden layers are sufficient to nearly achieve convergence rates of $\mc{O}(N^{-1/10})$, $\mc{O}(N^{-3/10})$, and $\mc{O}(N^{-1/2})$, respectively. 
%

%% file: main.bbl
\begin{bibdiv}
\begin{biblist}

\bib{Achour2022generalLplbd}{article}{
      author={Achour, E.~M.},
      author={Foucault, A.},
      author={Gerchinovitz, S.},
      author={Malgouyres, F.},
       title={A general approximation lower bound in ${L}^{p}$ norm, with
  applications to feed-forward neural networks},
        date={2022},
     journal={Advances in Neural Information Processing Systems},
      volume={35},
       pages={22396\ndash 22408},
}

\bib{Artin1964gamma}{book}{
      author={Artin, E.},
       title={The {G}amma function},
      series={Athena Series: Selected Topics in Mathematics},
   publisher={Holt, Rinehart and Winston},
        date={1964},
}

\bib{Bach2017Breaking}{article}{
      author={Bach, F.},
       title={Breaking the curse of dimensionality with convex neutral
  networks},
        date={2017},
     journal={J. Mach. Learn. Res.},
      volume={18},
}

\bib{Ball:1986cube}{article}{
      author={Ball, K.},
       title={Cube slicing in {$\mathbf{R}^n$}},
        date={1986},
     journal={Proc. Amer. Math. Soc.},
      volume={97},
      number={3},
       pages={465\ndash 473},
}

\bib{Barron:1992}{article}{
      author={Barron, A.~R.},
       title={Neural net approximation},
        date={1992},
     journal={Proc. 7th Yale Workshop on Adaptive and Learning Systems},
      volume={1},
       pages={69\ndash 72},
}

\bib{Barron:1993}{article}{
      author={Barron, A.~R.},
       title={Universal approximation bounds for superpositions of a sigmoidal
  function},
        date={1993},
     journal={IEEE Trans. Inform. Theory},
      volume={39},
      number={3},
       pages={930\ndash 945},
}

\bib{Barron:1994}{article}{
      author={Barron, A.~R.},
       title={Approximation and estimation bounds for artificial neural
  networks},
        date={1994},
     journal={Mach. Learn.},
      volume={14},
       pages={115\ndash 133},
}

\bib{Bresler:2020}{article}{
      author={Bresler, G.},
      author={Nagaraj, D.},
       title={A corrective view of neural networks: Representation,
  memorization and learning},
        date={2020},
     journal={Proceedings of Machine Learning Research},
      volume={125},
       pages={848\ndash 901},
}

\bib{BreslerNagaraj:2020}{article}{
      author={Bresler, G.},
      author={Nagaraj, D.},
       title={Sharp representation theorems for {ReLU} networks with precise
  dependence on depth},
        date={2020},
     journal={Advances in Neural Information Processing Systems},
      volume={33},
       pages={10697\ndash 10706},
}

\bib{Caragea:2023}{article}{
      author={Caragea, A.},
      author={Petersen, P.},
      author={Voigtlaender, F.},
       title={Neural network approximation and estimation of classifiers with
  classification boundary in a {B}arron class},
        date={2023},
     journal={Ann. Appl. Probab.},
      volume={33},
      number={4},
       pages={3039\ndash 3079},
}

\bib{DeVore2021NNapproximation}{article}{
      author={DeVore, R.},
      author={Hanin, B.},
      author={Petrova, G.},
       title={Neural network approximation},
        date={2021},
     journal={Acta Numer.},
      volume={30},
       pages={327\ndash 444},
}

\bib{Donahue:1997}{article}{
      author={Donahue, M.~J.},
      author={Gurvits, L.},
      author={Darken, C.},
      author={Sontag, E.},
       title={Rates of convex approximation in non-{H}ilbert spaces},
        date={1997},
        ISSN={0176-4276,1432-0940},
     journal={Constr. Approx.},
      volume={13},
      number={2},
       pages={187\ndash 220},
         url={https://doi.org/10.1007/s003659900038},
      review={\MR{1437210}},
}

\bib{E2018analytic}{article}{
      author={E, W.},
      author={Wang, Q.},
       title={Exponential convergence of the deep neural network approximation
  for analytic functions},
        date={2018},
     journal={Sci. China Math.},
      volume={61},
      number={10},
       pages={1733\ndash 1740},
}

\bib{Bolcskei:2021}{article}{
      author={Elbr\"achter, D.},
      author={Perekrestenko, D.},
      author={Grohs, P.},
      author={B\"olcskei, H.},
       title={Deep neural network approximation theory},
        date={2021},
     journal={IEEE Trans. Inform. Theory},
      volume={67},
      number={5},
       pages={2581\ndash 2623},
}

\bib{Eldan:2016}{article}{
      author={Eldan, R.},
      author={Shamir, O.},
       title={The power of depth for feedforward neural networks},
        date={2016},
     journal={Proceedings of Machine Learning Research},
      volume={49},
       pages={907\ndash 940},
}

\bib{haagerup1981best}{article}{
      author={Haagerup, U.},
       title={The best constants in the {K}hintchine inequality},
        date={1982},
     journal={Studia Math.},
      volume={70},
      number={3},
       pages={231\ndash 283},
}

\bib{Hormander:1963}{book}{
      author={H\"ormander, L.},
       title={Linear partial differential operators},
      series={Grundlehren der mathematischen Wissenschaften},
   publisher={Springer Berlin, Heidelberg},
        date={1963},
      volume={116},
}

\bib{Jiao2023Holder}{article}{
      author={Jiao, Y.},
      author={Lai, Y.},
      author={Lu, X.},
      author={Wang, F.},
      author={Yang, J.~Z.},
      author={Yang, Y.},
       title={Deep neural networks with {ReLU}-sine-exponential activations
  break curse of dimensionality in approximation on hölder class},
        date={2023},
     journal={SIAM J. Math. Anal.},
      volume={55},
      number={4},
       pages={3635\ndash 3649},
}

\bib{Barron:2018}{article}{
      author={Klusowski, J.~M.},
      author={Barron, A.~R.},
       title={Approximation by combinations of {ReLU} and squared {ReLU} ridge
  functions with $\ell^1$ and $\ell^0$ controls},
        date={2018},
     journal={IEEE Trans. Inform. Theory},
      volume={64},
      number={12},
       pages={7649\ndash 7656},
}

\bib{LiaoMing:2023}{article}{
      author={Liao, Y.},
      author={Ming, P.},
       title={Spectral {B}arron space for deep neural network approximation},
        date={accepted},
     journal={SIAM J. Math. Data Sci.},
        note={arXiv:2309.00788},
}

\bib{LuShenYangZhang:2021}{article}{
      author={Lu, J.},
      author={Shen, Z.},
      author={Yang, H.},
      author={Zhang, S.},
       title={Deep network approximation for smooth functions},
        date={2021},
     journal={SIAM J. Math. Anal.},
      volume={53},
      number={5},
       pages={5465\ndash 5506},
}

\bib{Luke:1962}{book}{
      author={Luke, Y.~L.},
       title={{Integrals of Bessel Functions}},
   publisher={McGraw-Hill Book Company},
        date={1962},
}

\bib{MaSiegelXu:2022}{article}{
      author={Ma, L.},
      author={Siegel, J.~W.},
      author={Xu, J.},
       title={Uniform approximation rates and metric entropy of shallow neural
  networks},
        date={2022},
     journal={Res. Math. Sci.},
      volume={9},
      number={3},
}

\bib{Makovoz1996Random}{article}{
      author={Makovoz, Y.},
       title={Random approximants and neural networks},
        date={1996},
     journal={J. Approx. Theory},
      volume={85},
      number={1},
       pages={98\ndash 109},
}

\bib{MengMing:2022}{article}{
      author={Meng, Y.},
      author={Ming, P.},
       title={A new function space from {B}arron class and application to
  neural network approximation},
        date={2022},
     journal={Commun. Comput. Phys.},
      volume={32},
      number={5},
       pages={1361\ndash 1400},
}

\bib{DGSM60:2014}{book}{
      author={Mikhailets, V.~A.},
      author={Murach, A.~A.},
       title={{H\"ormander Spaces, Interpolation, and Elliptic Problems}},
      series={De Gruyter Studies in Mathematics},
   publisher={Walter de Gruyter GmbH, Berlin/Boston},
        date={2014},
      volume={60},
}

\bib{Du:2019}{article}{
      author={Montanelli, H.},
      author={Du, Q.},
       title={New error bounds for deep {ReLU} networks using sparse grids},
        date={2019},
     journal={SIAM J. Math. Data Sci.},
      volume={1},
      number={1},
       pages={78\ndash 92},
}

\bib{Montanelli2021bandlimited}{article}{
      author={Montanelli, H.},
      author={Yang, H.},
      author={Du, Q.},
       title={Deep {R}e{LU} networks overcome the curse of dimensionality for
  generalized bandlimited functions},
        date={2021},
     journal={J. Comput. Math.},
      volume={39},
      number={6},
       pages={801\ndash 815},
}

\bib{Pilipovic:2010}{article}{
      author={Pilipovi\'c, S.},
      author={Teofanov, N.},
      author={Toft, J.},
       title={Micro-local analysis in {F}ourier {L}ebesgue and modulation
  spaces: part {II}},
        date={2010},
     journal={J. Pseudo-Differ. Oper. Appl.},
      volume={1},
      number={3},
       pages={341\ndash 376},
}

\bib{Pinkus1999}{article}{
      author={Pinkus, A.},
       title={Approximation theory of the {MLP} model in neural networks},
        date={1999},
     journal={Acta Numer.},
      volume={8},
       pages={143\ndash 195},
}

\bib{Patrick2020Lecture}{misc}{
      author={Rebeschini, P.},
       title={Lecture notes in algorithmic foundations of learning: {C}overing
  numbers bounds for {R}ademacher complexity. chaining},
        date={2020},
  url={https://www.stats.ox.ac.uk/~rebeschi/teaching/AFoL/22/material/lecture05.pdf},
}

\bib{Shamir:2022}{article}{
      author={Safran, I.},
      author={Eldan, R.},
      author={Shamir, O.},
       title={Depth separations in neural networks: What is actually being
  separated?},
        date={2022},
     journal={Constr. Approx.},
      volume={55},
      number={1},
       pages={225\ndash 257},
}

\bib{Shen2022Optimal}{article}{
      author={Shen, Z.},
      author={Yang, H.},
      author={Zhang, S.},
       title={Optimal approximation rate of {R}e{LU} networks in terms of width
  and depth},
        date={2022},
     journal={J. Math. Pures Appl. (9)},
      volume={157},
       pages={101\ndash 135},
}

\bib{Siegel2025optimal}{misc}{
      author={Siegel, J.~W.},
       title={Optimal approximation of zonoids and uniform approximation by
  shallow neural networks},
        date={2025},
        note={arXiv:2307.15285},
}

\bib{Siegel:2020}{article}{
      author={Siegel, J.~W.},
      author={Xu, J.},
       title={Approximation rates for neural networks with general activation
  functions},
        date={2020},
     journal={Neural Networks},
      volume={128},
       pages={313\ndash 321},
}

\bib{Siegel:2022}{article}{
      author={Siegel, J.~W.},
      author={Xu, J.},
       title={High-order approximation rates for shallow neural networks with
  cosine and {ReLU}$^k$ activation functions},
        date={2022},
     journal={Appl. Comput. Harmon. Anal.},
      volume={58},
       pages={1\ndash 26},
}

\bib{SiegelXu:2022}{article}{
      author={Siegel, J.~W.},
      author={Xu, J.},
       title={Sharp bounds on the approximation rates, metric entropy, and
  $n$-widths of shallow neural networks},
        date={2024},
     journal={Found. Comput. Math.},
      volume={24},
       pages={481\ndash 537},
}

\bib{Stein:1971}{book}{
      author={Stein, E.~M.},
      author={Weiss, G.},
       title={{Introduction to Fourier Analysis on Euclidean Spaces}},
      series={Princeton Mathematical Series},
   publisher={Princeton University Press},
        date={1971},
}

\bib{Telgarsky:2016}{article}{
      author={Telgarsky, M.},
       title={Benefits of depth in neural networks},
        date={2016},
     journal={Proceedings of Machine Learning Research},
      volume={49},
       pages={1517\ndash 1539},
}

\bib{Suzuki:2021}{article}{
      author={Tsuji, K.},
      author={Suzuki, T.},
       title={Estimation error analysis of deep learning on the regression
  problem on the variable exponent {B}esov space},
        date={2021},
     journal={Electron. J. Stat.},
      volume={15},
      number={1},
       pages={1869\ndash 1908},
}

\bib{Xu:2020}{article}{
      author={Xu, J.},
       title={Finite neuron method and convergence analysis},
        date={2020},
     journal={Commun. Comput. Phys.},
      volume={28},
      number={5},
       pages={1707\ndash 1745},
}

\bib{Yarotsky2017error}{article}{
      author={Yarotsky, D.},
       title={Error bounds for approximations with deep {ReLU} networks},
        date={2017},
     journal={Neural networks},
      volume={94},
       pages={103\ndash 114},
}

\bib{Yarotsky2018}{inproceedings}{
      author={Yarotsky, D.},
       title={Optimal approximation of continuous functions by very deep {ReLU}
  networks},
        date={2018},
   booktitle={Proceedings of the 31st conference on learning theory},
      series={Proc. Mach. Learn. Res.},
      volume={75},
   publisher={PMLR},
       pages={639\ndash 649},
}

\bib{Yarotsky2020phase}{inproceedings}{
      author={Yarotsky, D.},
      author={Zhevnerchuk, A.},
       title={The phase diagram of approximation rates for deep neural
  networks},
        date={2020},
   booktitle={Advances in neural information processing systems},
      volume={33},
   publisher={Curran Associates, Inc.},
       pages={13005\ndash 13015},
}

\bib{Yukich1995Supnorm}{article}{
      author={Yukich, J.~E.},
      author={Stinchcombe, M.~B.},
      author={White, H.},
       title={Sup-norm approximation bounds for networks through probabilistic
  methods},
        date={1995},
     journal={IEEE Trans. Inform. Theory},
      volume={41},
      number={4},
       pages={1021\ndash 1027},
}

\end{biblist}
\end{bibdiv}
